\documentclass[12pt, reqno,draft]{amsart}
\usepackage{amsmath}
\usepackage{amsthm}
\usepackage{amssymb}
\usepackage[abbrev]{amsrefs}
\usepackage{mathrsfs}
\usepackage[usenames]{color}
\usepackage{bm}
\usepackage{bbm}
\usepackage{enumitem}
\newtheorem{thm}{}[section]
\newtheorem{theorem}[thm]{Theorem}

\newtheorem{lemma}[thm]{Lemma}
\newtheorem{proposition}[thm]{Proposition}
\theoremstyle{definition}
\newtheorem{definition}[thm]{Definition}
\theoremstyle{remark}
\newtheorem{remark}[thm]{Remark}

\numberwithin{equation}{section}
\allowdisplaybreaks
\newcommand{\abs}[1]{\left\lvert#1\right\rvert}
\newcommand{\norm}[1]{\left\lVert#1\right\rVert}
\newcommand{\floor}[1]{\left\lfloor #1 \right\rfloor}
\newcommand{\ceil}[1]{\left\lceil #1 \right\rceil}

\newcommand{\SL}{\ensuremath{\mathcal{L}}}
\newcommand{\bes}{\ensuremath{\bm{\beta}}}
\newcommand{\hil}{\ensuremath{\bm{\eta}}}
\newcommand{\EE}{\ensuremath{\mathbb{E}}}
\newcommand{\Ind}{\ensuremath{\mathbbm{1}}}
\newcommand{\pot}{\ensuremath{\bm{\psi}}}
\newcommand{\sldf}{\ensuremath{\bm{\varphi_l^s}}}
\newcommand{\sudf}{\ensuremath{\bm{\varphi_u^s}}}

\newcommand{\udf}{\ensuremath{\bm{\varphi_u}}}
\newcommand{\kl}{\ensuremath{\widetilde{\bm{k}}}}
\newcommand{\kk}{\ensuremath{\bm{k}}}
\newcommand{\Fou}{\ensuremath{\mathcal{F}}}

\newcommand{\EB}{\ensuremath{\mathcal{E}}}
\newcommand{\dd}{\ensuremath{\mathbf{d}}}
\newcommand{\DD}{\ensuremath{\mathcal{D}}}
\newcommand{\ww}{\ensuremath{\bm{w}}}
\newcommand{\uu}{\ensuremath{\bm{u}}}

\newcommand{\FF}{\ensuremath{\mathbb{F}}}
\newcommand{\Sym}{\ensuremath{\mathbb{S}}}
\newcommand{\LL}{\ensuremath{\mathcal{L}}}
\newcommand{\Id}{\operatorname{Id}}
\newcommand{\NN}{\ensuremath{\mathbb{N}}}

\newcommand{\BX}{\ensuremath{\mathbb{B}}}
\newcommand{\XX}{\ensuremath{\mathbb{X}}}
\newcommand{\YY}{\ensuremath{\mathbb{Y}}}
\newcommand{\ZZ}{\ensuremath{\mathbb{Z}}}
\newcommand{\Disc}{\ensuremath{\mathbb{D}}}
\newcommand{\yy}{\ensuremath{\bm{y}}}
\newcommand{\bb}{\ensuremath{\bm{b}}}
\newcommand{\YB}{\ensuremath{\mathcal{Y}}}
\newcommand{\ZB}{\ensuremath{\mathcal{Z}}}
\newcommand{\BB}{\ensuremath{\mathcal{B}}}
\newcommand{\Fu}{\ensuremath{\mathcal{F}}}
\newcommand{\Cu}{\ensuremath{\mathcal{C}}}
\newcommand{\ee}{\ensuremath{\bm{e}}}
\newcommand{\WW}{\ensuremath{\mathbb{W}}}
\newcommand{\vv}{\ensuremath{\bm{v}}}
\newcommand{\VV}{\ensuremath{\mathbb{V}}}

\newcommand{\XB}{\ensuremath{\mathcal{X}}}
\newcommand{\VB}{\ensuremath{\mathcal{V}}}
\newcommand{\xx}{\ensuremath{\bm{x}}}
\newcommand{\leb}{\ensuremath{\bm{L}}}
\newcommand{\zz}{\ensuremath{\bm{z}}}

\newcommand{\supp}{\operatorname{supp}}
\newcommand{\Av}{\operatorname{Av}}
\AtBeginDocument{\def\MR#1{}}

\begin{document}
\title[Existence of almost greedy bases in Besov spaces]{Existence of almost greedy bases in mixed-norm sequence and matrix spaces, including Besov spaces}
\author[F. Albiac]{Fernando Albiac}\address{Department of Mathematics, Statistics, and Computer Sciences and INAMAT$^2$\\ Universidad P\'ublica de Navarra\\
Pamplona 31006\\ Spain}
\email{fernando.albiac@unavarra.es}

\author[J. L. Ansorena]{Jos\'e L. Ansorena}\address{Department of Mathematics and Computer Sciences\\
Universidad de La Rioja\\
Logro\~no 26004\\ Spain}
\email{joseluis.ansorena@unirioja.es}

\author[G. Bello]{Glenier Bello}
\address{Departamento de Matem\'{a}ticas\\
Universidad Aut\'{o}noma de Madrid\\
28049 Madrid;
ICMAT CSIC-UAM-UC3M-UCM\\
28049 Madrid,
Spain}
\email{glenier.bello@uam.es}

\author[P. Wojtaszczyk]{Przemys{\l}aw Wojtaszczyk}
\address{Institute of Mathematics of the Polish Academy of Sciences\\
00-656 Warszawa\\
ul.\ \'Sniadeckich 8\\
Poland}
\email{wojtaszczyk@impan.pl}
\begin{abstract}
We prove that the sequence spaces $\ell_p\oplus\ell_q$ and the spaces of infinite matrices $\ell_p(\ell_q)$, $\ell_q(\ell_p)$ and $(\bigoplus_{n=1}^\infty \ell_p^n)_{\ell_q}$, which are isomorphic to certain Besov spaces, have an almost greedy basis whenever $0<p<1<q<\infty$. More precisely, we custom-build almost greedy bases in such a way that the Lebesgue parameters grow in a prescribed manner. Our arguments critically depend on the extension of the Dilworth-Kalton-Kutzarova method from \cite{DKK2003}, which was originally designed for constructing almost greedy bases in Banach spaces, to make it valid for direct sums of mixed-normed spaces with nonlocally convex components. Additionally, we prove that the fundamental functions of all almost greedy bases of these spaces grow as $(m^{1/q})_{m=1}^\infty$.
\end{abstract}
\maketitle
\section{Introduction}\label{Introduction}\noindent
The concepts of greedy basis, almost greedy basis and quasi-greedy basis sprang from the study of the efficiency of the thresholding greedy algorithm (TGA for short) relative to bases in Banach spaces. Both greedy and quasi-greedy bases were introduced in the pioneering work of Konyagin and Temlyakov \cite{KoTe1999} from 1999, whereas almost greedy bases were defined shortly afterwards by Dilworth et al.\ \cite{DKKT2003} in what with hindsight would be, together with the work of Wojtaszczyk \cite{Woj2000}, the forerunner article on the functional analytic approach to the theory. If Konyagin and Temlyakov had characterized greedy bases as unconditional bases with the additional property of being democratic, Dilworth et al.\ characterized almost greedy bases as those bases that are simultaneously quasi-greedy and democratic.

An important research topic in functional analysis and abstract approximation theory is to determine if a given space has one of the three above-mentioned types of greedy-like bases.
It is clear that a Banach space cannot have a greedy basis unless it has an unconditional basis and this rules out many natural examples such as $L_{1}[0,1]$ and $\Cu[0,1]$. One can also give other examples failing to have a greedy basis such as $\ell_{1}\oplus \ell_{2}$; here it is a classical result of Edelstein and Wojtaszczyk \cite{EdWo1976} stating that any normalized unconditional basis is equivalent to the canonical basis, which is plainly non democratic.

However, it is much easier for a Banach space $\XX$ to have an almost greedy basis and a general construction was provided by Dilworth, Kalton, and Kutzarova in \cite{DKK2003}. Their technique, called for short the DKK-method, starts from the assumption that $\XX$ has a basis and then suppose that it has a complemented subspace $\Sym$ with a symmetric basis, so that $\XX$ is isomorphic to $\XX\oplus \Sym$. Then it is possible to construct a new basis of the direct sum which behaves very like the symmetric basis of $\Sym$, in the sense that inherits many of the good properties of the symmetric basis. More specifically, the following theorem was proved in \cite{DKK2003}.

\begin{theorem}\label{dkkmethod}
Let $\XX$ be a Banach space with a (Schauder) basis and suppose that $\XX$ has a complemented subspace with a symmetric basis $\Sym$. Then:
\begin{enumerate}[label=(\alph*)]
\item If $\Sym$ is not isomorphic to $c_{0}$ then $\XX$ has a quasi-greedy basis.
\item If $\Sym$ has nontrivial cotype then $\XX$ has an almost greedy basis.
\end{enumerate}
\end{theorem}

For instance, Theorem~\ref{dkkmethod} permits to readily deduce that $L_{1}[0,1]$ does have an almost greedy basis (see \cite{Gogyan2010} for a constructive example).

Once we know that a Banach space has an almost greedy basis, there are two possible directions to make headway in the theory. The first one is to estimate the size of the Lebesgue parameters $(\leb_{m})_{m=1}^{\infty}$ (see the definition in Section~\ref{sec:Lebesgue}) as a way to investigate the efficiency of the TGA relative to the basis, or, put in other words, how far the basis is from being greedy. The article \cite{AADK2019b} gathers some of the most significative advances in this direction. However, since some techniques from that paper rely heavily on the local convexity of the space, the problem of existence of almost greedy bases has remained opened in many important classical spaces, like the ones described in the title. Our contribution in this note closes this gap in the theory. It also connects with the other possible course of action by providing estimates for the Lebesgue parameters of the almost greedy bases whose existence we prove, and determine their prospective fundamental functions.

To be precise, our study includes, among other spaces, the finite direct sums
\begin{equation}\label{dpq}
D_{p,q}:=\ell_p\oplus\ell_q, \end{equation}
the matrix spaces
\begin{equation}\label{zpq}
Z_{p,q}:=\ell_q(\ell_p), \end{equation}
and the mixed-norm spaces of the family
\begin{equation}\label{bpq}
B_{p,q}:=\left(\bigoplus_{n=1}^\infty \ell_p^{n}\right)_q,
\end{equation}
for $0<p,q\le \infty$, with the convention that $\ell_{\infty}$ means $c_{0}$. Note that we use $(\oplus_{n=1}^\infty \XX_n)_q$ to denote the direct sum of the quasi-Banach spaces $\XX_n$ in the $\ell_q$-sense (in the $c_0$-sense if $q=\infty$).

The conductive thread of Section~\ref{DKKMethod} is the search for results that will allow us to construct almost greedy bases in certain nonlocally convex spaces while in Section~\ref{sec:Lebesgue} we study their Lebesgue parameters. We also enhance some results from \cite{AADK2019b} and tailor almost greedy bases whose Lebesgue parameters have a prescribed growth order (versus getting Lebesgue parameters whose growth is controlled only from below).

In Section~\ref{Main} we apply our techniques to show that the spaces $Z_{p,q}$, $B_{p,q}$, and $D_{p,q}$ for $0<p<1<\infty$ possess almost greedy bases.
We recall that the matrix spaces $Z_{p,q}$ are isomorphic to Besov spaces over Euclidean spaces (see, e.g., \cite{AlbiacAnsorena2016b}) and that the mixed-norm spaces $B_{p,q}$ are isomorphic to Besov spaces over the unit interval (see, e.g., \cite{AlbiacAnsorena2017}*{Appendix 4.2}).

Apart from the trivial cases, namely
\[
D_{q,q}\simeq Z_{q,q}\simeq B_{q,q}\simeq \ell_q, \quad 0< q\le\infty,
\]
and the case
\begin{equation}\label{Peuchinski}
\ell_q \simeq B_{2,q}, \quad 1<q<\infty,
\end{equation}
all the above-mentioned spaces are mutually non-isomorphic (see \cite{AlbiacAnsorena2017}).
The isomorphism in \eqref{Peuchinski} was obtained by Pe{\l}czy{\'n}ski in \cite{Pel1960} by combining the uniform complementation of the embeddings of $\ell_2^n$ into $\ell_q^{2^n}$, $n\in\NN$, with Pe{\l}czy{\'n}ski's decomposition technique (see, e.g., \cite{AlbiacKalton2016}*{Theorem 2.2.3}). Another well-known consequence of Pe{\l}czy{\'n}ski's decomposition technique is that for any unbounded sequence of integers $(d_n)_{n=1}^\infty$ we have
\[
B_{p,q}\simeq(\oplus_{n=1}^\infty \ell_p^{d_n})_q, \quad 0<p,q\le \infty,
\]
(see, e.g., \cite{AlbiacAnsorena2017}*{Appendix 4.1}).

In Section~\ref{sec:FF}, we study the fundamental functions of all almost greedy bases of the spaces $Z_{p,q}$, $B_{p,q}$, and $D_{p,q}$. Here, it is worth point out that the results we obain shed light to the locally convex cases and yield in particular all the possible fundamental functions of the spaces $Z_{p,q}$ when $p,q\ge 1$.

\subsection{Terminology.} Throughout this paper we will use standard notation and terminology from Banach spaces and greedy approximation theory, as can be found, e.g., in \cite{AlbiacKalton2016}. We also refer the reader to the recent article \cite{AABW2021} for other more specialized notation. We next single out however the most heavily used terminology.

The underlying field of our spaces will be the real or complex field, which will be denoted by $\FF$. For the set of all scalars of modulus $1$ we will use $\EE$.

The modulus of concavity $\kappa=\kappa[\XX]$ of a quasi-Banach space $\XX$ is the smallest constant $C$ such that $\norm{f+g}\le C (\norm{f} +\norm{g})$ for all $f$ and $g\in\XX$. More generally, given $b\in\NN$, we denote by $\kappa_b=\kappa_b[\XX]$ the optimal constant $C$ such that
\[
\norm{\sum_{j=1}^b f_j} \le C \sum_{j=1}^b \norm{f_j}, \quad f_j\in\XX.
\]
In this notation, $\kappa=\kappa_2$.

Let $\EB=(\ee_n)_{n=1}^\infty$ be the unit vector system of $\FF^\NN$, and let $\EB^*=(\ee_n^*)_{n=1}^\infty$ be its coordinate functionals. A \emph{sequence space} is a quasi-Banach space $\Sym\subseteq\FF^\NN$ for which $\EB$ is a Schauder basis. Given a sequence space $\Sym$ and $N\in\NN$, we denote by $\Sym^{(N)}$ the $N$-dimensional space spanned by $(\ee_n)_{n=1}^N$ regarded as a subspace of $\Sym$. If the truncation applies to a sequence $\XB=(\xx_n)_{n=1}^\infty$ in a quasi-Banach space $\XX$, we set $\XB^{(N)}=(\xx_n)_{n=1}^N$.
Given a collection of signs $\varepsilon=(\varepsilon_n)_{n\in A}\in\EE^A$, we set
\[
\Ind_{\varepsilon,A}[\XB,\XX]=\sum_{n\in A} \varepsilon_n\, \xx_n.
\]
If $\varepsilon_n=1$ for all $n\in A$, we set $\Ind_{A}[\XB,\XX]=\Ind_{\varepsilon,A}[\XB,\XX]$.

The direct sum of $N$ Schauder bases $\XB_k=(\xx_{k,n})_{n=1}^\infty$ of quasi-Banach spaces $\XX_k$, $1\le k\le N$, is the Schauder basis $\XB=\bigoplus_{k=1}^{N}\XB_k=(\xx_j)_{j=1}^\infty$ of $\XX=\bigoplus_{k=1}^{N}\XX_k$ constructed by alternating the elements of each basis; that is, if $L_k$ denotes the canonical embedding of $\XX_k$ into $\XX$ then, if $L_k\colon \XX_k \to \XX$ is the canonical embedding,
\[
\xx_{(n-1)N+k}=L_k(\xx_{k,n}), \quad n\in\NN, \; 1\le k\le N.
\]

For each $k\in\NN$ let $\XB_k=(\xx_{k,n})_{n=1}^{N_k}$ be a Schauder basis of a finite-dimensional quasi-Banach space $\XX_k$. Suppose the basis constants of the bases $\XB_k$ and the moduli of concavity of the quasi-Banach spaces $\XX_k$ are uniformly bounded, and
let $\LL$ be a quasi-Banach lattice over $\NN$. The infinite direct sum $(\bigoplus_{k=1}^\infty \XB_k)_\LL$ of the bases $\XB_k$ is the Schauder basis $(\xx_j)_{j=1}^\infty$ of $\XX=(\bigoplus_{k=1}^\infty \XX_k)_\LL$ constructed by concatenating the bases. That is, if $M_n=\sum_{k=1}^n N_k$ and $L_k$ is the canonical embedding of $\XX_k$ into $\XX$, then
\[
\xx_{M_{k-1}+n}=L_k(\xx_{k,n}), \quad k\in\NN, \; 1\le n\le N_k.
\]

On occasion will use upper and lower estimates in quasi-Banach lattices. Let us recall the definitions.
A quasi-Banach lattice $\XX$ is said to satisfy an \emph{upper $r$-estimate} (resp., lower $r$-estimate), $0<r<\infty$, if there is a constant $C$ such that, for every pairwise disjoint finite family $\Fu=(f_j)_{j\in J}$ in $\XX$, $S(\Fu)\le C T(\Fu)$ (resp., $T(\Fu)\le C S(\Fu)$), where
\[
S(\Fu)=\norm{\sum_{j\in J} f_j}, \quad T(\Fu)=\left(\sum_{j\in J} \norm{ f_j}^r\right)^{1/r}.
\]

A quasi-Banach lattice $\LL$ over $\NN$ is said to be \emph{minimal} if $c_{00}$ is dense in $\LL$.

Other more specific notation will be introduced in context, when needed.
\section{The Dilworth-Kalton-Kutzarova method for quasi-Banach spaces}\label{DKKMethod}\noindent
The DKK-method was invented in \cite{DKK2003} and subsequently widely studied in \cite{AADK2019b} with the aim to construct bases of Banach spaces with some special features which are of interest in abstract approximation theory. The prototypical ingredients one needs in order to be able to implement this method are:
\begin{itemize}[leftmargin=*]
\item An \emph{ordered partition} of $\NN$, i.e., a sequence $(\sigma_n)_{n=1}^\infty$ of disjoint subsets of $\NN$, with union equal to $\NN$, such that $\max(\sigma_{n})<\min(\sigma_{n+1})$ for all $n\in\NN$;
\item a semi-normalized Schauder basis $\XB=(\xx_n)_{n=1}^\infty$ of a Banach space $\XX$; and
\item a locally convex \emph{subsymmetric} sequence space $\Sym\subseteq\FF^\NN$, i.e., a Banach sequence space such that for all $f=\sum_{n=1}^{\infty}a_n\ee_n\in\Sym$, all $(\varepsilon_n)_{n=1}^\infty\in\EE^\NN$, and all $\phi\colon\NN\to\NN$ increasing we have
\[
\norm{\sum_{n=1}^{\infty}\varepsilon_na_n\ee_{\phi(n)}}=\norm{f}.
\]
\end{itemize}

Ordered partitions consist of integer intervals. To be precise, we have $\sigma_n=[1+M_{n-1},M_n]$ for all $n\in\NN$, where $M_0=0$ and $M_n=\sum_{j=1}^n\abs{\sigma_j}$ for all $n\in\NN$.

Given $f=(a_n)_{n=1}^\infty\in\FF^\NN$ and $A\subseteq\NN$ finite, we put
\[
\Av(f,A)= \frac{1}{|A|} \left(\sum_{j\in A} a_j\right).
\]
The \emph{averaging projection} with respect to an ordered partition $\sigma=(\sigma_n)_{n=1}^\infty$ is the map $P_\sigma\colon \FF^\NN\to\FF^\NN$ defined by
\[
(a_j)_{j=1}^\infty\mapsto (b_k)_{k=1}^\infty, \quad b_k =\Av(f,\sigma_n) \mbox{ if }k\in \sigma_n.
\]

The boundedness of the averaging projections (see \cite{LinTza1977}*{Propostion 3.a.4}) on $\Sym$ is a central component for the method to work, so we cannot do without the assumption that $\Sym$ is locally convex. In this section we will enhance the DKK-method by dropping the premise that $\XX$ is a Banach space so that it also permits to obtain \emph{special} bases of nonlocally convex spaces.

Let $\Lambda[\Sym]=(\Lambda_m[\Sym])_{m=1}^\infty$ be the fundamental function of the unit vector system of $\Sym$, i.e.,
\[
\Lambda_m[\Sym]=\norm{\Ind_{\varepsilon,A}}_\Sym, \quad \abs{A}=m, \, \varepsilon\in\EE^A.
\]
Let $\Lambda^*[\Sym]=(\Lambda^*_m[\Sym])_{m=1}^\infty$ be its dual sequence,
\[
\Lambda^*_m[\Sym]=\frac{m}{\Lambda_m[\Sym]}, \quad m\in\NN.
\]

We next consider sequences $\VB_{\Sym,\sigma}=(\vv_n)_{n=1}^\infty$ and $\VB^*_{\Sym,\sigma}=(\vv_n^*)_{n=1}^\infty$ given by
\begin{equation}\label{BOSforAP}
\vv_n=\frac{1}{\Lambda_{N_n}} \Ind_{\sigma_n}[\EB,\Sym] \quad\mbox{ and }\quad
\vv_n^*= \frac{1}{\Lambda^*_{N_n}} \Ind_{\sigma_n}[\EB^*,\Sym^*],
\end{equation}
where $N_n=\abs{\sigma_n}$. The sequence $\VB_{\Sym,\sigma}$ is a normalized basic sequence in $\Sym$, and $\VB^*_{\Sym,\sigma}$ is a semi-normalized basic sequence in $\Sym^*$ (see \cite{LinTza1977}*{Propostion 3.a.6}) biorthogonal to $\VB$. With this terminology, the averaging projection $P_\sigma$ can be expressed as
\[
P_\sigma(f) =\sum_{n=1}^\infty \vv_n^*(f)\vv_n, \quad f\in\FF^\NN.
\]
We denote by $Q_{\sigma}=\Id_{\FF^\NN}-P_{\sigma}$ the complementary projection of $P_\sigma$. Since $\norm{P_\sigma}_{\Sym\to\Sym}\le 2$ we infer that $\norm{Q_\sigma}_{\Sym\to\Sym}\le 3$.

Let us also introduce the map
\[
H_{\XB,\XX,\sigma}\colon c_{00} \to \XX, \quad f\mapsto \sum_{n=1}^\infty \vv_n^*(f)\, \xx_n.
\]

If $\sigma$ is an ordered partition of $\NN$, $\Sym$ is a locally convex subsymmetric sequence space, and $\XB$ is a semi-normalized Schauder basis of a quasi-Banach space, the formula
\[
\norm{f}_{\XB,\Sym,\sigma}
=\norm{ Q_\sigma (f)}_\Sym+\norm{H_{\XB,\XX,\sigma} (f)}_\XX,
\]
clearly defines a semi-quasi-norm on $c_{00}$, which ends up being a quasi-norm. This can be deduced from the equivalence
\begin{equation}\label{eq:A}
\norm{f}_{\XB,\XX,\sigma} \simeq \norm{f}_\Sym, \quad n\in\NN,\; \supp(f)\subseteq\sigma_n,
\end{equation}
whose proof goes over the lines of the proof of that equivalence in the locally convex setting (see \cite{AADK2019b}*{Lemma 3.3}). Estimate~\eqref{eq:A} also yields the existence of a constant $C$ such that
\begin{equation*}
\abs{\ee_j^*(f)} \le C\norm{ f }_{\XB,\Sym,\sigma}, \quad j\in\NN,\; f\in\FF^\NN.
\end{equation*}
Therefore, the completion of $c_{00}$ equipped with the quasi-norm $\norm{\cdot }_{\XB,\XX,\sigma}$ falls inside $c_0$.

\begin{definition}
Let $\sigma$ be an ordered partition of $\NN$, $\Sym$ be a locally convex subsymmetric sequence space, and let $\XB$ be a semi-normalized Schauder basis of a quasi-Banach space $\XX$. We define the sequence space $\YY[\XB,\Sym,\sigma]$ as the completion of $(c_{00},\norm{ \cdot}_{\XB,\Sym,\sigma})$.
\end{definition}

The next result gathers some properties of the space $\YY[\XB,\Sym,\sigma]$ and its unit vector system $\EB$.

\begin{theorem}\label{thm:FirstProperties}
Let $\XB=(\xx_j)_{j=1}^\infty$ be a Schauder basis for a quasi-Banach space $\XX$, let $\Sym$ be a locally convex subsymmetric sequence space, and let $\sigma=(\sigma_n)_{n=1}^\infty$ be an ordered partition of $\NN$. Denote $N_n=\abs{\sigma_n}$ for $n\in\NN$ and $M_r=\sum_{n=1}^r N_n$ for $r\in \NN$. We have:
\begin{enumerate}[label=(\roman*), leftmargin=*,widest=iii]
\item\label{FP:1} The unit-vector system is a semi-normalized Schauder basis for $\YY=\YY[\XB,\Sym,\sigma]$.
\item\label{FP:2} $\YY \simeq Q_\sigma(\Sym) \oplus \XX$.
\item\label{FP:3} $\YY^{(M_r)}\simeq Q_\sigma(\Sym^{(M_r)})\oplus [\xx_j \colon 1\le j \le r]$ uniformly on $r\in\NN$.
\item A basis $\XB'$ for a Banach space $\XX'$ dominates $\XB$ if and only if
$\YY[\XB',\Sym,\sigma] \subseteq \YY$ continuously.
\end{enumerate}
\end{theorem}

\begin{proof}
The proof in the locally convex case (see \cite{AADK2019b}*{Theorem~3.6}) can be extended verbatim to the case when $\XX$ is a quasi-Banach space.\qedhere
\end{proof}

We go deeper in the understanding of the sequence space $\YY[\XB,\Sym,\sigma]$ by studying the most relevant greedy-like properties of its unit vector system.

\subsection{Democracy} Recall that a sequence $\Gamma=(\Gamma_m)_{m=1}^\infty$ in $(0,\infty)$ is said to have the \emph{lower regularity property} (LRP for short) if there is a positive integer $b$ such
\[
2\Gamma_m \le \Gamma_{bm}, \quad m\in\NN.
\]
In turn, we say that $\Gamma$ has the \emph{upper regularity property} (URP for short) if there a positive integer $b$ such
\[
2\Gamma_{bm} \le b \Gamma_{m}, \quad m\in\NN.
\]
In the case when both $\Gamma$ and its dual sequence $\Gamma^*=(m/\Gamma_m)_{m=1}^\infty$ are nondecreasing, then $\Gamma$ has the LRP if and only if $\Gamma^*$ has the URP.

It is known (see \cite{AADK2019b}*{Corollary 3.10}) that if $\Sym$ is locally convex, the sequence $\Lambda[\Sym]$ has the LRP, and the size of the partition $(\sigma_n)_{n=1}^\infty$ `grows steadily', then the unit vector system of $\YY[\XB,\XX,\sigma]$ is democratic. The proof of this result does not pass on to nonlocally convex spaces, however. To obtain the democracy of the unit vector system of $\YY[\XB,\Sym,\sigma]$ in the more general setting of quasi-Banach spaces we must impose additional conditions and use a different argument. We get started with a regularity lemma.

\begin{lemma}\label{lem:regularityrstimate}
Let $(N_n)_{n=1}^\infty$ be a sequence of positive integers such that $M_n:=\sum_{k=1}^n N_k \lesssim N_{n+1}$ for $n\in\NN$, let $(\Gamma_n)_{n=1}^\infty$ be a sequence of positive scalars having the LRP, and let $0<p\le 1$. Assume that $(\Gamma_n)_{n=1}^\infty$ and $(n/\Gamma_n)_{n=1}^\infty$ are non-decreasing. Then
\[
\sup_{r\in\NN} \sum_{n=1}^r \frac{\Gamma^p_{N_n}}{\Gamma^p_{N_r}} <\infty\; \mbox{ and }\;
\sup_{r\in\NN} \sum_{n=r}^\infty \frac{\Gamma^p_{N_r}}{\Gamma^p_{N_n}} <\infty.
\]
\end{lemma}

\begin{proof}
An appeal to \cite{AlbiacAnsorena2016}*{Lemma 2.12} yields constants $0<\alpha<1$ and $0<C_1<\infty$ such that
\[
\frac{\Gamma_n}{n^\alpha} \le C_1 \frac{\Gamma_m}{m^\alpha}\quad \mbox{for all } n\le m.
\]
Let $C_2>1$ be such that $M_r:=\sum_{n=1}^r N_n\le C_2 N_r$ for all $r\in\NN$. Then, if $t=C_2/(1+C_2)$ we have
$M_{r+1}\le t M_r $ for $r\in\NN$. Hence,
\begin{align*}
\sum_{n=1}^r \frac{\Gamma^p_{N_n}}{\Gamma^p_{N_r}}
& \le \sum_{n=1}^r \frac{\Gamma^p_{M_n}}{\Gamma^p_{N_r}}\\
& \le \sum_{n=1}^r \frac{M_r^p}{N_r^p} \frac{\Gamma^p_{M_n}}{\Gamma^p_{M_r}}\\
& \le C_2^p\sum_{n=1}^r \frac{M_r^p}{M_n^p} \frac{\Gamma^p_{M_n}}{\Gamma^p_{M_r}}\\
& \le C_1^p C_2^p \sum_{n=1}^r \left(\frac{M_n}{M_r}\right)^{p(1- \alpha)}\\
&\le C_1^p C_2^p \sum_{n=1}^r t^{p (1-\alpha)(n-r)}\\
&\le C_1^p C_2^p \frac{1}{1-t^{p(1-\alpha)}}.
\end{align*}

Similarly,
\begin{align*}
\sum_{n=r}^\infty \frac{\Gamma^p_{N_r}}{\Gamma^p_{N_n}}
& \le \sum_{n=r}^\infty\frac{M_n^p}{N_n^p} \frac{\Gamma^p_{N_r}}{\Gamma^p_{M_n}}\\
& \le \sum_{n=r}^\infty\frac{M_n^p}{N_n^p} \frac{\Gamma^p_{M_r}}{\Gamma^p_{M_n}}\\
& \le C_2^p \sum_{n=r}^\infty \frac{\Gamma^p_{M_r}}{\Gamma^p_{M_n}}\\
&\le C_1^p C_2^p \sum_{n=r}^\infty \left(\frac{M_r}{M_n}\right)^{p \alpha}\\
&\le C_1^p C_2^p \sum_{n=r}^\infty t^{p \alpha(n-r)}\\
&= C_1^p C_2^p \frac{1}{1-t^{p\alpha}},
\end{align*}
as desired.
\end{proof}

From now on, a \emph{weight} will be a non-negative sequence $(w_n)_{n=1}^\infty$ with $w_1>0$. Given $0< q<\infty$ and a weight $\ww=(w_n)_{n=1}^\infty$, the \emph{Lorentz sequence space} $d_q(\ww)$ consists of all sequences $f=(a_n)_{n=1}^\infty$ in $c_{0}$ whose non-increasing rearrangement $(a_n^*)_{n=1}^\infty$ verifies
\[
\norm {f }_{d_q(\ww)}= \left( \sum_{n=1}^\infty ( s_n a_n^*)^q \frac{w_n}{s_n} \right)^{1/q} <\infty,
\]
where $s_n=\sum_{k=1}^n w_k$. In turn, the \emph{weak Lorentz sequence space} $d_\infty(\ww)$ consists of all sequences $f=(a_n)_{n=1}^\infty$ in $c_0$ whose non-increasing rearrangement $(a_n^*)_{n=1}^\infty$ satisfies
\[
\norm{ f}_{d_\infty(\ww)}=\sup_n a_n^* s_n<\infty.
\]
We have $\Lambda_m[d_q(\ww)] \approx s_m$ for $m\in\NN$. Moreover if $0< p\le q\le \infty$,
\begin{equation*}
\norm {f }_{d_q(\ww)} \lesssim \norm {f }_{d_p(\ww)}, \quad f\in c_0.
\end{equation*}
Although the notation of Lorentz sequence spaces highlights the weight $\ww$, it has to be made explicit that they rather depend on the primitive sequence $(s_m)_{m=1}^\infty$, as the following result shows.
\begin{lemma}[see \cite{AABW2021}*{\S9}]\label{lem:LorentzEq}
Let $\ww=(w_n)_{n=1}^\infty$ and $\ww'=(w_n')_{n=1}^\infty$ be weights, and let $0<q\le\infty$. Then $\norm {f}_{d_q(\ww)} \approx \norm {f }_{d_p(\ww')}$ for $f\in c_0$ if and only if $\sum_{n=1}^m w_n \approx \sum_{n=1}^m w_n'$ for $m\in\NN$.
\end{lemma}
The reader will find in \cite{AABW2021}*{\S9} the background needed on Lorentz sequence spaces. Our interest in these spaces lays in their applicability to describing certain features of bases which we specify next.

A quasi-Banach space $\YY$ is said to be \emph{squeezed between two symmetric sequence spaces $\Sym_1$ and $\Sym_2$ via a biorthogonal system} $(\yy_n,\yy_n^*)_{n=1}^\infty$ if the \emph{series transform}, given by
\[
(a_n)_{n=1}^\infty\mapsto \sum_{n=1}^\infty a_n\, \yy_n
\]
is bounded from $\Sym_1$ into $\XX$, and the \emph{coefficient transform}
\[
f\mapsto \Fou(f):=(\yy_n^*(f))_{n=1}^\infty
\]
is bounded from $\XX$ into $\Sym_2$. We say that $\YY$ is squeezed between $\Sym_1$ and $\Sym_2$ via a Schauder basis $\YB$ with coordinate functionals $\YB^*$ if it is squeezed between $\Sym_1$ and $\Sym_2$ via the biorthogonal system $(\YB,\YB^*)$. It the case when $\YB$ is the unit vector system of a sequence space $\YY$, this is equivalent to having continuous embeddings
\[
\Sym_1\subseteq \YY\subseteq \Sym_2.
\]
If the squeezing spaces $\Sym_1$ and $\Sym_2$ are close to each other in the sense that
\[
\Lambda[\Sym_1] \approx \Lambda[\Sym_2],
\]
we say the the biorthogonal system is \emph{squeeze-symmetric}. A sequence (commonly a Schauder basis) $(\yy_n)_{n=1}^\infty$ in $\YY$ will be said to be squeeze-symmetric if there is $(\yy_n^*)_{n=1}^\infty$ in $\YY^*$ such that the biorthogonal system $(\yy_n,\yy_n^*)_{n=1}^\infty$ is squeeze-symmetric.

If $(\yy_n)_{n=1}^\infty$ is squeeze-symmetric, in particular it is democratic. In fact, if we introduce the \emph{upper super-democracy function} of a basis (also called the \emph{fundamental function}), given for $m\in \NN$ by
\[
\sudf[\XB,\XX](m)=\sup\left\{ \norm{\Ind_{\varepsilon,A}[\XB,\XX]} \colon \abs{A} \le m, \, \varepsilon\in\EE^A\right\},
\]
and the \emph{lower super-democracy function} of the basis, defined for $m\in\NN$ by
\[
\sldf[\XB,\XX](m)=\inf\left\{ \norm{\Ind_{\varepsilon,A}[\XB,\XX]} \colon \abs{A} \ge m, \, \varepsilon\in\EE^A\right\},
\]
then the fundamental function of a squeeze-symmetric basis is equivalent to the fundamental function of the symmetric sequence spaces which sandwich the space $\YY$.

It is known that if a basis $\YB$ is squeeze-symmetric and $\YY$ is locally $r$-convex, $0<r\le 1$, then we can choose $\Sym_1=d_r(\ww)$ and $\Sym_2=d_\infty(\ww)$ as squeezing spaces, where the primitive weight of $\ww$ is equivalent to the fundamental function of $\YB$ (see \cite{AABW2021}*{Theorem~9.14}). We observe that the same result, with the same proof, holds for biorthogonal systems.

In the proof of the following theorem we use the concept of doubling function. Recall that a sequence $\Gamma=(\Gamma_m)_{m=1}^\infty$ in $(0,\infty)$ is said to be \emph{doubling} if there is a constant $C$ such that $\Gamma_{2m} \le C \Gamma_m$ for all $m\in\NN$. The \emph{dual sequence} of $\Gamma$ is the sequence $\Gamma^*=(m/\Gamma_m)_{m=1}^\infty$. If $\Gamma^*$ is non-decreasing, then $\Gamma_{bm} \le b\Gamma_m$ for all $b$ and $m\in\NN$, so $\Gamma$ is doubling.

\begin{theorem}\label{thm:SqSy}
Assume that all the hypotheses of Theorem~\ref{thm:FirstProperties} hold, that $\Lambda[\Sym]$ has both the LRP and the URP, and that $M_r \lesssim N_{r+1}$ for $r\in\NN$. Then the unit vector system of $\YY=\YY[\XB,\Sym,\sigma]$ is squeeze-symmetric. In particular it is democratic and its fundamental function grows as $\Lambda[\Sym]$.
\end{theorem}

\begin{proof}
Set $\Lambda[\Sym]=(\Lambda_n)_{n=1}^\infty$ and $\ww=(\Lambda_n - \Lambda_{n-1})_{n=1}^\infty$. A close look at the proof of \cite{AADK2019b}*{Theorem~3.9} reveals that the only information we need about $\XB$ and $\XX$ in order to obtain the embedding $\YY\subseteq d_\infty(\ww)$ is that the coefficient transform maps $\XX$ into $c_0$. Therefore this embedding still holds in the case when $\XX$ is a quasi-Banach space. Thus, by \cite{AABW2021}*{Corollary 9.13}, it suffices to prove that $\sudf[\EB,\YY](m) \lesssim \Lambda_m$ for $m\in\NN$. To that end we will use that the sequence $\Lambda^*[\Sym]=(\Lambda^*_m)_{m=1}^\infty$ has the LRP.

Without loss of generality we may suppose that $\XX$ is a $p$-Banach space for some $0<p\le 1$. Set $c=\sup_n \norm{\xx_n}$. Since $\Lambda$ is doubling, there is a constant $C$ such that
\[
\Lambda_{M_r} \le C \Lambda_{1+M_{r-1}}, \quad r\in\NN.
\]
By Lemma~\ref{lem:regularityrstimate} there is a constant $C_1$ such that
\[
\sum_{n=1}^{r-1} \Lambda_{N_n}^p \le \frac{C_1^p}{\Lambda_{N_{r-1}}^p} \quad \mbox{and} \quad
\sum_{n=r}^\infty \frac{1}{(\Lambda_{N_n}^*)^p} \le \frac{C_1^p}{\Lambda_{N_r}^p}, \quad r\in\NN.
\]

Given $m\in\NN$, choose $r\in\NN$ such that $M_{r-1}< m \le M_{r}$. Pick $A\subseteq\NN$ with $\abs{A}\le m$, and $\varepsilon\in\EE^A$. We have $\abs{A\cap \sigma_n} \le \min\{m,N_n\}$ for all $n\in\NN$. Hence,

\begin{align*}
\norm{H_{\Sym,\XX,\sigma}(\Ind_{\varepsilon,A})}^p
&\le c^p \sum_{n=1}^\infty \abs{\vv_n^*(\Ind_{\varepsilon,A})}^p\\
&\le c^p\sum_{n=1}^\infty \frac{\abs{A\cap\sigma_n}^p}{(\Lambda_{N_n}^*)^p}\\
&\le c^p\sum_{n=1}^{r-1} \frac{N_n^p}{(\Lambda_{N_n}^*)^p}+ c^p m^p\sum_{n=r}^\infty \frac{1}{(\Lambda_{N_n}^*)^p}\\
&\le c^p\sum_{n=1}^{r-1} \Lambda_{N_n}^p+ c^p N_r^p\sum_{n=r}^\infty \frac{1}{(\Lambda_{N_n}^*)^p}\\
&\le \Lambda_{N_{r-1}}^p + c^p C_1^p\frac{N_r^p}{(\Lambda_{N_r}^*)^p}\\
&\le c^p C_1^p (1+C) \Lambda_m.
\end{align*}

In turn,
\[
\norm{Q_\sigma(\Ind_{\varepsilon,A})}_\Sym\le 3 \norm{ \Ind_{\varepsilon,A}}_\Sym\le 3 \Lambda_m.\qedhere
\]
\end{proof}

\subsection{Quasi-greediness}\label{DKKQG}
Once we know that a basis is democratic, in order to prove that it is almost greedy we must show that it is quasi-greedy (see \cite{AABW2021}*{Theorem 6.3}). To address this task, we will take advantage of the vestiges of unconditionality retained by squeeze-symmetric bases.

\begin{lemma}\label{lem:Dem+TQG}
Let $\YB=(\yy_n)_{n=1}^\infty$ be a squeeze-symmetric basis of a quasi-Banach space $\YY$ with coordinate functionals $(\yy_n^*)_{n=1}^\infty$. Then there is a constant $C$ such that $\norm{f}\le C\norm{g}$ whenever $f$, $g\in\YY$ satisfy
\[
\abs{\supp(f)}\le \abs{\{ n\in\NN \colon \max_{s\in\NN} \abs{\yy_s^*(f)} \le \abs{\yy_n^*(g)}\}}.
\]
\end{lemma}

\begin{proof}
Suppose that $\YY$ is squeezed between $\Sym_1$ and $\Sym_2$, and that $\Lambda[\Sym_1]\approx \Lambda[\Sym_2]$. Let $f$ and $g$ be as in the statement of the lemma. Set $t= \max_{n\in\NN} \abs{\yy_n^*(f)}$ and $m= \abs{\{ n\in\NN \colon t \le \abs{\yy_n^*(g)}\}}$. We have
\[
\norm{f} \lesssim \norm{\Fou(f)}_{\Sym_1}
\le t \Lambda_{m}[\Sym_1]
\approx t \Lambda_m[\Sym_2]
\le \norm{\Fou(g)}_{\Sym_2} \lesssim \norm{g}.\qedhere
\]
\end{proof}

We notice that, as a matter of fact, the property in Lemma~\ref{lem:Dem+TQG} characterizes squeeze-symmetric bases. Next, we introduce an asymptotic unconditionality property which combined with squeeze-symmetry produces almost greedy bases.

If $\XB=(\xx_n)_{n=1}^\infty$ is a Schauder basis of $\XX$ with coordinate functionals $\XB^*=(\xx_n^*)_{n=1}^\infty$, and $A$ is a subset of $\NN$, the \emph{coordinate projection} onto the closed linear span $[\xx_{n}\colon n\in A]$ is the linear operator
\[
S_A[\XB,\XX]\colon \XX\to \XX, \quad f\mapsto \sum_{n\in A} \xx_n^*(f) \, \xx_n.
\]
A basis $\YB=(\yy_n)_{n=1}^\infty$ of a quasi-Banach space $\XX$ is \emph{quasi-greedy} if there is a constant $C$ such that $\norm{S_A(f)} \le C \norm{f}$ for every $f\in\XX$ and every greedy set $A$ of $f$, i.e., $\abs{\yy_k^*(f)}\le\abs{\yy_j^*(f)}$ for all $j\in A$ and all $k\in\NN\setminus A$. Since greedy sets depend on the function $f$, quasi-greediness could be regarded as a nonlinear property. The following definition singles out a linear property which combined with squeeze-symmetry, implies quasi-greediness, as shown in Proposition~\ref{prop:SqSyAU} below.

\begin{definition}\label{def:AU}
Let $\XB$ be a Schauder basis of a quasi-Banach space. We say that $\XB$ is \emph{asymptotically suppression unconditional} if there exists a constant $D$ such that $\norm{S_A}\le D$ whenever $\abs{A} <\min A$.
\end{definition}

Next we show that we can relax the condition in Definition~\ref{def:AU} to obtain an equivalent property to asymptotic suppression unconditionality.

\begin{lemma}\label{lem:CASU}
Let $\XB$ be a Schauder basis of a quasi-Banach space. Suppose that there and $D\in(0,\infty)$, $b\in\NN$, and $d\in\NN$ such that $\norm{S_A}\le D$ whenever $\abs{A}> d$ and $ b \abs{A}<\min A$. Then $\XB$ is asymptotically suppression unconditional.
\end{lemma}

\begin{proof}
If $\abs{A}\le d$ then $\norm{S_A}\le \kappa_d C$, where $C=\sup_n \norm{S_{\{n\}}} <\infty$. Therefore the result holds in the case when $b=1$. Hence to prove the general case it suffices to check that $\norm{S_A}$ remains bounded when $A$ runs over all sets with $\abs{A} > \max\{2b,d\}$ and $ \abs{A} <\min A$. Given one such sets, put $m=\ceil{\abs{A}/(2b)}$, so that we can split $A$ into $2b$ sets $(A_j)_{j=1}^{2b}$ of cardinality at most $m$. Since
\[
b\abs{A_j}\le bm \le b \left( \frac{\abs{A}}{2b} +1 \right) =\frac{\abs{A}}{2} +b<\abs{A}<\min A\le \min A_j,
\]
it follows that $\norm{S_{A_j}}\le D$ for $j=1$, \dots $2b$ and so $\norm{S_A} \le \kappa_{2b} D$.
\end{proof}

\begin{proposition}\label{prop:SqSyAU}
Let $\YB=(\yy_n)_{n=1}^\infty$ be a squeeze-symmetric asymptotically unconditional basis of a quasi-Banach space $\YY$. Then $\YB$ is quasi-greedy.
\end{proposition}

\begin{proof}

Given a greedy set $A$ of $f\in\XX$, put $t=\min\{\abs{\yy_n^*(f)}\colon n\in A\}$, $m=\abs{A}$, and $F=[1,m]\cap\ZZ$. Set $B=A_m\setminus A$ and $E=A\setminus F$. Since $(A\cap F, B)$ is a partition of $F$ and $(A\cap F, E)$ is a partition of $A$, we can decompose the projection onto $A$ in the form
\[
S_A(f)=S_{F}(f)+S_E(f)-S_B(f).
\]
Since $\YB$ is a Schauder basis, $\norm{S_{F}(f)} \le K\norm{f}$, where $K$ is the basis constant. We have $\abs{E} \le m$, and $ m< \min E$. So, by asymptotic suppression unconditionality, $\norm{S_{E}(f)} \le D\norm{f}$. Also, $\abs{\yy_n^*(f)}\le t$ for all $n\in B$, $\abs{\yy_n^*(f)}\ge t$ for all $n\in A$, and $\abs{B}\le \abs{A}$. Thus applying Lemma~\ref{lem:Dem+TQG} (where $S_B(f)$ plays the role of $f$ and $f$ the role of $g$) we get $\norm{S_B(f)}\le C \norm{f}$. Summing up,
\[
\norm{S_A(f)} \le \kappa_3\left(\norm{S_{F}(f)}+\norm{S_{E}(f)} +\norm{S_{B}(f)}\right)\le \kappa_3 ( K+D+ C) \norm{f}.\qedhere
\]
\end{proof}

The following two lemmas reflect the interplay between coordinate projections and averaging projections.
\begin{lemma}[\cite{AADK2019b}*{Lemma 3.11}]\label{lem: projectionlemma}
Let $(\Sym,\Vert \cdot\Vert_\Sym)$ be a subsymmetric sequence space and $\sigma=(\sigma_n)_{n=1}^\infty$ be an ordered partition of $\NN$. Set $N_n=\abs{\sigma_n}$, $\Lambda^*[\Sym]=(\Lambda_m^*)_{m=1}^\infty$, and $\VB^*_{\Sym,\sigma}=(\vv_n^*)_{n=1}^\infty$. Then there is a constant $C$ such that
\[
\abs{\vv_n^*(S_A(f))}\le C \frac{\Lambda^*_{m}}{\Lambda^*_{N_n}} \norm{ S_{\sigma_n}(f) }_\Sym
\]
for all $m\in\NN$, all $n\in\NN$, all $f\in c_{00}$, and all $A\subseteq \sigma_n$ with $\abs{A}\le m$.
\end{lemma}

\begin{lemma}\label{lem:Axnormestimate}
Let $\Sym$ be a locally convex subsymmetric sequence space and $\sigma=(\sigma_n)_{n=1}^\infty$ be an ordered partition of $\NN$. Then:
\begin{enumerate}[label=(\roman*),leftmargin=*,widest=iii]
\item\label{it:DilworthA} $\norm{P_\sigma(S_A(f))}_\Sym \le 4 \norm{f}_\Sym$ for all $f\in Q_\sigma(\Sym)$.
\item\label{it:DilworthB}
If $\Sym$ satisfies an upper $s$-estimate for some $s\ge 1$, there is a constant $C$ such that
\[
\norm{Q_\sigma(S_A(f))}_\Sym \le 5 \norm{Q_\sigma(f)}_\Sym+ C \left( \sum_{n= 1}^\infty \frac{\Lambda^s_{m_n}} {\Lambda^s_{N_n}} \abs{\vv_n^*(f)}^s\right)^{1/s},
\]
for all $A \subseteq\NN $ and $f \in\Sym $, where $m_n = \abs{A \cap \sigma_n}$.
\end{enumerate}
\end{lemma}

\begin{proof}
The proof of \cite{AADK2019b}*{Proof of Lemma~3.14} gives both $\ref{it:DilworthA}$ and also that $\norm{Q_\sigma(S_A(f))}_\Sym \le 5 \norm{f}_\Sym$ for all $f\in Q_\sigma(f)$. It remains to prove $\ref{it:DilworthB}$ when $f\in Q_\sigma(\Sym)$. To that end, we go over the lines of \cite{AADK2019b}*{Lemma~3.14}, replacing the triangle law with the upper $s$-estimate.
\end{proof}

\begin{lemma}\label{lem:newregular}
Let $(N_n)_{n=1}^\infty$ be a sequence of positive integers such that $M_n:=\sum_{k=1}^n N_k \lesssim N_{n+1}$ for $n\in\NN$, let $(\Gamma_n)_{n=1}^\infty$ be a sequence of positive scalars having the LRP, and let $0<q\le 1$. Assume that $(\Gamma_n)_{n=1}^\infty$ and $(n/\Gamma_n)_{n=1}^\infty$ are non-decreasing. Then, there is a constant $C$ such that
\[
\sum_{n=r}^\infty \left( \frac {\Gamma_{m_n}} {\Gamma_{N_n}}\right)^q\le C^q
\]
for all $r\in\NN$ and all sequences $(m_n)_{n=r}^\infty$ in $\NN$ with $m_n\le M_r$ for all $n\ge r$.
\end{lemma}

\begin{proof}
Notice that $\Gamma_{bm} \le b \Gamma_m$ for all $b$ and $m\in\NN$. Besides, our hypothesis gives an integer $b$ such that $M_r\le b N_r$ for all $r\in\NN$. Hence,
\[
\sum_{n=r}^\infty \left( \frac {\Gamma_{m_n}} {\Gamma_{N_n}}\right)^q\le \sum_{n=r}^\infty \left( \frac {\Gamma_{b N_r}} {\Gamma_{N_n}}\right)^q
\le b^q \sum_{n=r}^\infty \left( \frac {\Gamma_{N_r}} {\Gamma_{N_n}}\right)^q.
\]
An appeal to Lemma~\ref{lem:regularityrstimate} puts an end to the proof.
\end{proof}

\begin{proposition}\label{prop:projectionnormestimate}
Assume that all the hypotheses of Theorem~\ref{thm:FirstProperties} hold, that $\Lambda[\Sym]$ has both the LRP and the URP, and that $M_r \lesssim N_{r+1}$ for $r\in\NN$. Then the unit vector system of $\YY=\YY[\XB,\XX,\sigma]$ is asymptotically suppression unconditional.
\end{proposition}
\begin{proof}
Assume without loss of generality that $\XX$ is a $p$-Banach space. Set $c=\sup_n\norm{\xx_n}$ and $c_*= \sup_n \norm{\xx_n^*}$. Notice that both $\Lambda[\Sym]=(\Lambda_m)_{m=1}^\infty$ and $\Lambda[\Sym^*]=(\Lambda_m^*)_{m=1}^\infty$ have the LRP. Let $C$ (resp.\ $C_*$) be the constant obtained by feeding Lemma~\ref{lem:newregular} with $\Gamma_m=\Lambda_m$ and $q=1$ (resp.\ $\Gamma_m=\Lambda_m^*$ and $q=p$). By \eqref{eq:A}, there is a constant $C_0$ such that
\begin{equation}\label{eq:B}
\norm{ S_{\sigma_n}(f)}_\Sym\le C_0 \norm{f}_{\XB,\Sym,\sigma}, \quad n\in\NN,\; f\in \YY.
\end{equation}
With an eye to applying Lemma~\ref{lem:CASU}, choose $b\in\NN$ such that $b \ge M_r/N_r$ for all $r\in\NN$. Pick $m\in\NN$ and $A\subset \NN$ such that $\abs{A}\le m$ and $bm<\min{A}$. Let $r\in\NN$ be such that $M_{r-1} <m\le M_r$. We have $\min{A}>M_{r-1}$. Hence, $A_n:=A\cap \sigma_n= \emptyset$ for all $1 \le n\le r-1$. Set $m_n=\abs{A_n}$ for $n\ge r$. Fix $f\in \YY$. By Lemma~\ref{lem: projectionlemma} and inequality \eqref{eq:B},
\begin{align*}
\norm{H_{\XB,\Sym,\sigma}(S_A(f))}_\XX^p
&\le c^p\sum_{n=r}^\infty \abs{\vv_n^*(S_{A_n}(f))}^p\\
& \le 2^p c^p \sum_{n=r}^\infty \left( \frac {\Lambda^*_{m_n}} {\Lambda^*_{N_n}}\right)^p \norm{ S_{\sigma_n}(f)}^p_\Sym\\
&\le 2^p c^p C_0^p \norm{f}_{\XB,\Sym,\sigma}^p \sum_{n=r}^\infty \left( \frac {\Lambda^*_{m_n}} {\Lambda^*_{N_n}}\right)^p\\
&\le 2^p c^p C_0^p C_*^p \norm{f}_{\XB,\Sym,\sigma}^p.
\end{align*}
Now, if $C_1$ is the constant provided by Lemma~\ref{lem:Axnormestimate}\ref{it:DilworthB} in the case $s=1$,
\begin{align*}
\norm{Q_\sigma(S_A(f))}_\Sym
&\le\max\left\{ 5 + c_* C_1 \sum_{n= r}^\infty \frac{\Lambda_{m_n}} {\Lambda_{N_n}} \right\} \norm{f}_{\XB,\XX,\sigma}\\
&\le\max\left\{ 5 + 2 c_* C_1 \right\} \norm{f}_{\XB,\XX,\sigma}.
\end{align*}
Combining we obtain $\norm{S_A(f)}_{\XB,\XX,\sigma}\le C_2 \norm{f}_{\XB,\XX,\sigma}$, where
\[
C_2= 2^p c^p C_0^p C_*^p + \max\left\{ 5 + c_* C C_1 \right\}.\qedhere
\]
\end{proof}

\begin{theorem}\label{thm:QGTheorem}
Assume that all the hypotheses of Theorem~\ref{thm:FirstProperties} hold, that $\Lambda[\Sym]$ has both the LRP and the URP, and that $M_r \lesssim N_{r+1}$ for $r\in\NN$. Then the unit vector system is an almost greedy basis for $\YY[\XB,\Sym,\sigma]$ with fundamental function equivalent to $\Lambda[\Sym]$.
\end{theorem}

\begin{proof}
It follows by combining Theorem~\ref{thm:SqSy} with Propositions~\ref{prop:SqSyAU} and \ref{prop:projectionnormestimate}.
\end{proof}

We close this section by noticing that, since we intend to apply Theorem~\ref{thm:QGTheorem} in the case when $\Sym$ is a superreflexive space, imposing regularity on $\Lambda[\Sym]$ is not a restriction.

\begin{proposition}\label{prop:RadLURP}
Let $\Sym$ be a superreflexive subsymmetric sequence space. Then $\Lambda[\Sym]$ has the URP and the LRP.
\end{proposition}

\begin{proof}
The unit vector system of $\Sym$ is almost greedy, and $\Sym$ has nontrivial type. Hence the result follows from
\cite{DKKT2003}*{Proposition 4.1}.
\end{proof}

\section{Estimates for the Lebesgue parameters of bases resulting from DKK-method}\label{sec:Lebesgue}\noindent
From the point of view of sparse approximation with respect to the TGA in Banach spaces, the most important numerical information one can get from a basis $\XB$ is obtained through
the sequence $(\leb_m)_{m=1}^\infty$ of its Lebesgue parameters.
For $m\in\NN$, the \emph{$m$th Lebesgue parameter} $\leb_m$ is defined as the smallest constant $C$ such that
\[
\norm{f-S_A(f)}\le C \norm{f-g}
\]
for all greedy sets $A$ of $f$ with $\abs{A}= m$ and all linear combinations $g$ of $m$ elements of $\XB$. Roughly speaking, the sequence $(\leb_m)_{m=1}^\infty$ measures how far $\XB$ is from being greedy.

The growth of the Lebesgue parameters is linearly determined by the combination of the unconditionality parameters $(\kk_m)_{m=1}^\infty$, defined as
\[
\kk_m=\kk_m[\XB,\XX]=\sup\left\{ \norm{ S_A[\XB,\XX]} \colon A\subseteq\NN, \ \abs{A} \le m \right\}, \quad m\in\NN,
\]
and which quantify the conditionality of $\XB$, and the squeeze-symmetric parameters, which quantify
how far $\XB$ is from being squeeze symmetric
(see \cite{AAB2021}*{Theorem 1.5}). In the case when the basis $\XB$ is squeeze-symmetric we have,
\[
\leb_{m} \approx \kk_{m}, \quad m\in\NN,
\]
hence for this type of bases the growth of the Lebesgue parameters is completely controlled by the growth of the unconditionality parameters.

As an instrument to get information on the growth of $(\kk_m)_{m=1}^\infty$ we will use a related sequence of unconditionality parameters given by
\[
\kl_m[\XB,\XX]=\sup\left\{ \norm{ S_A[\XB,\XX](f)} \colon f \in B_\XX \cap [\xx_j\colon 1\le j \le m], \, A\subseteq\NN\right\},
\]
where $B_\XX$ denotes the closed unit ball of $\XX$. Note that $\kl_m\le \kk_m$ for all $m\in\NN$.

Embeddings involving sequence spaces are an important tool in the theory to estimate the unconditionality parameters of a basis (see \cite{AAB2021}*{Section 5}). The most important embeddings using bases are the ones that involve the spaces $\ell_q$, $0<q\le \infty$. The notions of $q$-Hilbertian and $q$-Besselian bases were introduced to formalize this idea. We recall that a basis $\XB$ of a quasi-Banach space $\XX$ is said to be \emph{$q$-Hilbertian} if $\ell_q$ embeds in $\XX$ via the basis $\XB$; and the basis $\XB$ is said \emph{$q$-Besselian} if $\XX$ embeds in $\ell_q$ via $\XB$. In this paper, we will need to measure how far a basis $\XB$ of a quasi-Banach space $\XX$ is from being $q$-Hilbertian or $q$-Besselian. To that end, we introduce the parameters $\bes_r[\XB,\XX,q]$ and $\hil_r[\XB,\XX,q]$ below.

Given a basis $\XB$ of a quasi-Banach space $\XX$, $0<q\le \infty$, and $r\in\NN$, we denote by $\bes_r[\XB,\XX,q]$ the smallest constant $C$ such that
\[
\norm{\sum_{n=1}^r a_n \, \xx_n} \le C \left( \sum_{n=1}^r \abs{a_n}^q\right)^{1/q}, \quad a_n\in\FF,
\]
and by $\hil_r[\XB,\XX,q]$ the smallest constant $C$ such that
\[
\left( \sum_{n=1}^r \abs{a_n}^q\right)^{1/q} \le C \norm{\sum_{n=1}^r a_n \, \xx_n}, \quad a_n\in\FF.
\]

\begin{lemma}\label{lem:kl}
Assume that all the hypotheses of Theorem~\ref{thm:SqSy} hold, that $\Sym$ satisfies an upper $s$-estimate for some $s\ge 1$ and a lower $q$-estimate for some $q\le\infty$. Then there is a constant $C$ such that for all $r\in\NN$,
\[
\kl_r[\XB,\XX] \le \kl_{M_r}[\EB,\YY] \le C \max\{ \kl_r[\XB,\XX], \hil_r[\XB,\XX,s], \bes_r[\XB,\XX,q] \}.
\]
\end{lemma}

\begin{proof}
The proof of the left hand-side inequality goes over the lines of the proof of \cite{AADK2019b}*{Lemma 3.7}. To prove the right hand-side inequality we pick $r\in\NN$, $A\subseteq\NN$, and $f\in\YY$ with $\max( \supp(f)) \le M_r$. Note that this implies that $\supp(f)\subseteq\cup_{n=1}^r \sigma_r$. The vectors $g:=P_\sigma(f)$ and $h=:Q_\sigma(f)$ inherit this property from $f$. Since $\abs{\vv_n^*(S_A(g))}\le \abs{\vv_n^*(g)}$ for all $n\in\NN$, and $\vv_n^*(g)=0$ whenever $n>r$,
\[
\norm{H_{\XB,\XX,\sigma}(S_A(g))}_\XX\le C_1 \kl_r[\XB,\XX] \norm{H_{\XB,\XX,\sigma}(g)}_\XX,
\]
where $C_1$ is a geometric constant (see \cite{AABW2021}*{Corollary 2.4}). In turn, by Lemma~\ref{lem:Axnormestimate}\ref{it:DilworthA},
\begin{align*}
\norm{H_{\XB,\XX,\sigma}(S_A(h))}_\XX
&\le \bes_r[\XB,\XX,q]
\left(\sum_{n=1}^r \abs{\vv_n^*(S_A(h))}^q\right)^{1/q}\\
&\le C_2 \bes_r[\XB,\XX,q] \norm{P_\sigma(S_A(h))}_\Sym\\
&\le 4 C_2 \bes_r[\XB,\XX,q] \norm{P_\sigma(h)}_\Sym,
\end{align*}
where $C_2$ is the upper $C$-estimate constant of $\Sym$. Finally, if $C_3$ is the constant provided by Lemma~\ref{lem:Axnormestimate}\ref{it:DilworthB},
\begin{align*}
\norm{Q_\sigma(S_A(f))}_\Sym &\le 5 \norm{Q_\sigma(f)}_\Sym+C_3\left( \sum_{n=1}^r \abs{\vv_n^*(f)}^s\right)^{1/s}\\
&\le 5 \norm{Q_\sigma(f)}_\Sym+C_3 \hil_r[\XB,\XX,s] \norm{H_{\XB,\XX,\sigma}(f)}_\XX.
\end{align*}
Combining, we obtain the desired result.
\end{proof}

Lemma~\ref{lem:kl} provides an optimal estimate in the case when
\[
\max\{ \hil_r[\XB,\XX,s], \bes_r[\XB,\XX,q] \}\lesssim \kl_r[\XB,\XX], \quad r\in\NN.
\]
In this regard, we note that any semi-normalized Schauder basis of a quasi-Banach space $\XX$ satisfies the estimate
\[
\hil_m[\XB,\XX,s]\lesssim m^{1/s}, \quad r\in\NN,
\]
and that if $\XX$ is in addition locally $p$-convex,
\[
\bes_m[\XB,\XX,q]\lesssim m^{1/p-1/q}, \quad r\in\NN.
\]
Moreover, any superreflexive subsymmetric space satisfies an upper $s$-estimate for some $s>1$ and a lower $q$-estimate for some $q<\infty$. So, in order to apply Lemma~\ref{lem:kl}, it will be convenient to have \emph{highly conditional} Schauder bases for $\XX$, i.e., bases with $\kl_m \gtrsim m^{1/p}$ or, at least, $\kl_m \gtrsim m^a$ with $a$ arbitrarily close to $1/p$. In connection with this, we notice that the unconditionality parameters of any Schauder basis $\XB$ of a $p$-Banach space $\XX$ satisfy $\kk_m[\XB,\XX] \le m^{1/p}$, while if $\XX$ is a superreflexive Banach space we have $\kk_m[\XB,\XX] \lesssim m^a$ for some $a<1$ (see \cite{AAW2019}).

It is known that the $\ell_p$ spaces, $1<p<\infty$, have Schauder bases which are as highly conditional as desired (see \cite{GW2014}). To prove the existence of a highly conditional Schauder basis for $\ell_p$, $0<p\le 1$, we will make use of the \emph{difference system} $\DD=(\dd_n)_{n=1}^\infty$ defined for $n\in \NN$ by
\[
\dd_n=\ee_n -\ee_{n-1} \mbox{ (with the convention $\ee_0=0$)}.
\]

\begin{proposition}\label{prop:diflp}
Let $0<p<1$. The difference system $\DD=(\dd_n)_{n=1}^\infty$ is a semi-normalized Schauder basis of $\ell_p$ with $[\dd_n \colon 1 \le n \le N]=\ell_p^{(N)}$ for all $N\in\NN$ and
\[
\kl_m[\DD,\ell_p]=m^{1/p}, \quad m\in\NN.
\]
\end{proposition}

\begin{proof}
Pick $(a_n)_{n=1}^\infty\in c_{00}$. Notice that
\[
f:=\sum_{n=1}^\infty a_n\, \dd_n= \sum_{n=1}^\infty (a_n-a_{n+1})\, \ee_n,
\]
hence, given $m\in\NN$,
\[
S_m(f):=\sum_{n=1}^m a_n\, \dd_n= a_m \, \ee_m+\sum_{n=1}^{m-1} (a_n-a_{n+1})\, \ee_n.
\]
Since
\[
\abs{a_m} \le \abs{\sum_{n=m}^\infty a_n-a_{n+1} } \le \sum_{n=m}^\infty \abs{a_n-a_{n+1}} \le
\left( \sum_{n=m}^\infty \abs{a_n-a_{n+1}}^p\right)^{1/p},
\]
it follows that,
\[
\norm{S_m(f)}_p \le \left( \abs{a_m}^p+\sum_{n=1}^{m-1} \abs{a_n-a_{n+1}}^p\right)^{1/p}\le \norm{f}_p.
\]
This proves that $\DD$ is a monotone Schauder basis. To estimate its unconditionality parameters we note that for $m\in\NN$,
\[
\sum_{n=1}^m \dd_n =\ee_m,\quad \sum_{n=1}^m \dd_{2n-1}=\sum_{j=1}^{2m-1} (-1)^{j-1} \ee_j, \quad
\sum_{n=1}^m \dd_{2n}=\sum_{j=1}^{2m} (-1)^{j} \ee_j.
\]
We infer that $\kl_{m} \ge m^{1/p}$ for all $m\in\NN$. Since $[\dd_n \colon 1\le n \le N]=[\ee_n \colon 1\le n \le N]$, we are done.
\end{proof}

Lemma~\ref{lem:kl} alerts us of the role that the \emph{right inverse} $(B_m)_{m=1}^\infty$ of the sequence $(M_r)_{r=1}^\infty$, defined by
\[
B_m=\sup\{ r \in \NN \colon m\le M_r\},
\]
could play within the study of the conditionality parameters. The following technical lemma helps us in this direction.

\begin{lemma}\label{lem:AAW}
Let $\varphi\colon[0,\infty)\to[1,\infty)$ be a concave increasing function. There is an increasing sequence $(M_r)_{r=1}^\infty$ in $\NN$ such $M_r \lesssim M_{r+1}-M_{r}$ for $r\in\NN$, and whose right inverse $(B_m)_{m=1}^\infty$ satisfies
\[
B_m \approx \varphi(\log m), \quad m\in\NN.
\]
\end{lemma}

\begin{proof}
Let $\psi\colon[1,\infty) \to [0,\infty)$ be the inverse of $\varphi$. Choose $1<b<\infty$ such that $C:=b^{\psi(2)/2}>2$ and set
\[
M_r=\floor{ b^{\psi(r)}}, \quad r\in\NN.
\]
Since $\psi$ is convex, the map $t\mapsto \psi(t)/t$ is nondecreasing. Hence for $r\in\NN$,
\[
\frac{-1+b^{\psi(r+1)}}{b^{\psi(r)}}\ge -b^{-\psi(r)}+b^{\psi(r+1)/(r+1)}\ge -b^{-\psi(r)}+b^{\psi(2)/2}\ge C-1.
\]
We infer that $(C-1) M_r \le M_{r+1}$ for all $r\in\NN$. Therefore,
\[
M_r \le \frac{1}{C-2} (M_{r+1} - M_r), \quad r\in\NN.
\]
If $M_{r-1}<m\le M_r$, so that $r=B_m$, then $b^{\psi(r-1)} \le m \le b^{\psi(r)}$. Hence,
\[
-1+B_m \le \varphi(\log_b m) \le B_m, \quad m\in\NN.
\]
Since $\varphi$ is doubling, we are done.
\end{proof}

We close this section by proving that asymptotic unconditionality ensures that both kinds of uncontionality parameters are of the same order.

\begin{proposition}\label{prop:CCAU}
Let $\XB$ be an asymptotically suppression unconditional Schauder basis of a quasi-Banach space $\XX$. Then, $\kl_m[\XB,\XX] \approx\kk_m[\XB,\XX]$ for $m\in\NN$.
\end{proposition}

\begin{proof}
Let $D$ be as in Definition~\ref{def:AU}. Let $\kappa$ be the modulus of concavity of $\XX$ and let $K$ be the basis constant of $\XB$. For $m\in\NN$ fixed, pick $A\subseteq\NN$ with $\abs{A} \le m$ and set $B=[1,m]\cap \ZZ$, $E=A\setminus [1,m]$. For every $f\in\XX$ we have
\begin{align*}
\norm{S_A(f) } &\le \kappa ( \norm{S_A(S_B(f))}+ \norm{S_E(f)})\\
&\le \kappa( \kl_{m} \norm{S_B(f)} + D \norm{f})\\
&\le \kappa (K \kl_{m}+D)\norm{f}.
\end{align*}
Therefore, $\kl_m\le\kk_m \le \kappa (K \kl_{m}+D)$ for all $m\in\NN$, and the proof is over.
\end{proof}
\section{Quasi-Banach spaces with almost greedy bases}\label{Main}\noindent
In this section we prove the existence of almost greedy bases in certain quasi-Banach spaces. Our discussion applies, among others, to the mixed-norm spaces and matrix spaces built from $\ell_p$ spaces, $0<p\le \infty$ listed in \eqref{dpq}, \eqref{zpq}, and \eqref{bpq}.

In the cases when $p$, $q\in[1,\infty]$, the almost greedy basis structure of these spaces is well understood (see \cite{DKK2003} and \cite{AADK2019b}). The cases that correspond to indices $0<p,q<1$ are also completetly settled since the corresponding spaces have no almost greedy bases unless $p=q$ (see \cite{AABe2022}). Here, we complement this study by proving that the mixed-norm spaces $D_{p,q}$, the matrix spaces $Z_{p,q}$ and $Z_{q,p}$, and the Besov spaces $B_{p,q}$ have a very rich almost greedy basis structure in the case when $0<p<1<q<\infty$.

Our main theoretical result in this section is the following.

\begin{theorem}\label{thm:existence}
Let $\XX$ be a quasi-Banach space with a Schauder basis $\XB=(\xx_n)_{n=1}^\infty$, and let $\Sym$ be a superreflexive subsymmetric sequence space. Let $\delta\colon[0,\infty) \to [0,\infty)$ be a non-increasing map such that
\[
\kl_m=\kl_m[\XB,\XX]= \delta(m),\quad m\in\NN.
\]
Suppose that $\XX$ is $p$-convex, $0<p\le 1$. Let $1 <s <\infty$ be such that $\Sym$ satisfies an upper $s$-estimate, and let $1<q<\infty$ be such that $\Sym$ satisfies a lower $q$-estimate. Let $\varphi\colon[0,\infty) \to [0,\infty)$ be a concave increasing function. Then the space $\XX\oplus\Sym$ has an almost greedy Schauder basis $\BB=(\bb_n)_{n=1}^\infty$ with the following additional properties:
\begin{enumerate}[label=(\roman*),leftmargin=*,widest=iii]
\item\label{cond:dem} $\sudf[\BB,\XX\oplus\Sym]\approx \Lambda[\Sym]$.
\item $\BB$ is asymptotically suppression unconditional.
\item\label{cond:fdb} Put $\XX^{(N)}=[\xx_n \colon 1\le n \le N]$ and $\BX^{(N)}=[\bb_n \colon 1\le n \le N]$ for $N\in\NN$. There are increasing sequences $(K_r)_{r=1}^\infty$, $(L_r)_{r=1}^\infty$ and $(P_r)_{r=1}^\infty$ in $\NN$ such that, for $r\in\NN$,
\begin{enumerate}[label=(\arabic*),leftmargin=*]
\item $\XX^{(r)}\oplus\Sym^{(K_r)}$ is uniformly complemented in $\BX^{(L_r)}$, and
\item $\BX^{(L_r)}$ is uniformly complemented in $\XX^{(r)}\oplus\Sym^{(P_r)}$.
\end{enumerate}
\item\label{optionB} $\kl_{m}[\BB,\XX\oplus\Sym]\gtrsim \delta(\log m)$ for $m\ge 2$.
\end{enumerate}
Further, in the case when $\kl_m\gtrsim m^{\max\{1/s,1/p-1/q\}}$ for $m\in\NN$ we can choose $\BB$ satisfying
\begin{enumerate}[label=(\roman*),leftmargin=*,widest=iii,resume]
\item\label{optionA} $\kl_{m}[\BB,\XX\oplus\Sym]\approx \delta(\varphi(\log m))$ for $m\in\NN$
\end{enumerate}
instead of \ref{optionB}.
\end{theorem}

\begin{proof} We start by fixing a suitable ordered partition $\sigma=(\sigma_n)_{n=1}^\infty$. The way we choose that partition depends on whether we will use it to prove $(iv)$ or $(v)$. To tackle the proof of $(iv)$ we choose $\sigma$ satisfying
\begin{enumerate}[label=(\alph*),leftmargin=*]
\item\label{caseB} $M_r:=2^{r+1}$ for all $r\in\NN$, so that the right inverse $(B_m)_{m=1}^\infty$ of $(M_r)_{r=1}^\infty$ satisfies $B_m\approx 1+\log m$ for $m\in\NN$.
\end{enumerate}
In turn, to tackle the proof of $\ref{optionA}$, we use Lemma~\ref{lem:AAW} to pick $\sigma$ such that
\begin{enumerate}[label=(\alph*),leftmargin=*, resume]
\item\label{caseA}$M_r:=\sum_{n=1}^r \abs{\sigma_n} \lesssim
\abs{\sigma_{n+1}}$ for $r\in\NN$, and $B_m\approx \varphi(\log m)$ for $m\in\NN$.
\end{enumerate}

Now, let $\XB_0$ be the Schauder basis $\XX_0:=\XX\oplus P_\sigma(\Sym)$ given by
\[
\XB_0=\XB\oplus \VB_{\Sym,\sigma}.
\]
Combining Proposition~\ref{prop:projectionnormestimate}, Theorem~\ref{thm:QGTheorem}, Theorem~\ref{thm:existence} and Theorem~\ref{thm:FirstProperties}\ref{FP:1}, gives that the unit vector system is a asymptotically suppression unconditional almost greedy Schauder basis of $\YY=\YY[\XB,\Sym,\sigma]$ with $\sudf[\EB,\YY]\approx \Lambda[\Sym]$. By Theorem~\ref{thm:FirstProperties}\ref{FP:2},
\[
\YY\simeq \XX_0\oplus Q_\sigma(\Sym) \simeq \XX\oplus\Sym\oplus P_\sigma(\Sym) \oplus Q_\sigma(\Sym)\simeq \XX \oplus \Sym\oplus \Sym \simeq \XX \oplus \Sym.
\]
Hence, the basis $\BB$ that corresponds to $\EB$ via the above isomorphism from $\YY$ onto $\XX \oplus \Sym$ is an almost greedy Schauder basis of $\XX \oplus \Sym$ which satisfies $\ref{cond:dem}$ and $\ref{cond:fdb}$.

Let $\VV_{\Sym,\sigma}^{(r)}$ be the subspace of $\Sym$ spanned by the first $r$ vectors of $\VB_{\Sym,\sigma}$, $r\in\NN$. Notice that $\VV^{(r)}=P_\sigma(\Sym^{(M_r)})$. Hence, by Theorem~\ref{thm:FirstProperties}\ref{FP:3},
\[
\YY^{(M_{2r})}\simeq \XX^{(r)} \oplus \VV_{\Sym,\sigma}^{(r)} \oplus Q_\sigma ( \Sym^{(M_{2r})}) \simeq
\XX^{(r)} \oplus \Sym^{(M_r)} \oplus Q_\sigma(\WW_r),
\]
uniformly, where $\WW_r=[\ee_n \colon M_r < n \le M_{2r}]$. Since $Q_\sigma (\WW_r)$ is uniformly complemented in $ \Sym^{(M_{2r})}$, \ref{cond:fdb} holds.

The basis $\XB_0$ inherits from $\XB$ the property that $\kl_m[\XB_0,\XX_0]\approx \delta(m)$ for $m\in\NN$. By Lemma~\ref{lem:kl}, either $\kl_{m}[\EB,\YY]\gtrsim \delta(B_m-1)$ or
\[
\delta(B_m-1)\lesssim \kl_{m}[\EB,\YY] \lesssim \delta(B_m)
\]
for $m\in\NN$. Since $\sup_m \kl_{m+1}/\kl_m<\infty$, either \ref{optionB} or \ref{optionA} holds, depending on whether we are in case~\ref{caseB} or \ref{caseA}. \end{proof}

Since any Banach space with a basis has a conditional basis (\cite{PelSin1964}), Theorem~\ref{thm:existence} serves to construct a conditional almost greedy basis of $\XX\oplus\Sym$. However, the main interest of the theorem resides in proving that certain Banach spaces without a greedy basis possess, at least, an almost greedy basis, and, further, in estimating its unconditionality parameters.

In this regard, let us mention that if $0<p<1<q<\infty$, the mixed-norm spaces $D_{p,q}$ and the Besov spaces $Z_{p,q}$ and $Z_{q,p}$ do not have greedy bases because their canonical basis is the unique unconditional basis of the space up to permutation (see \cite{AABW2021}*{\S11}). Because of the way these spaces are defined, and how the norm of vectors written down in terms of their canonical bases is computed, at first it might appear to be shocking that these spaces could have a democratic basis at all. However, as we will see they possess almost greedy bases since all of them contain a complemented copy of the superreflexive space $\ell_q$. We will also prove the existence of almost greedy bases in the Besov spaces $B_{p,q}$ for $0<p<1<q<\infty$, where the existence of greedy bases remains open.

We note that the unconditionality parameters of any almost greedy basis $\YB$ of a $p$-Banach space $\XX$, $0<p\le 1$, satisfy the estimate
\[
\kk_m[\YB,\YY] \lesssim (1+\log m)^{1/p}, \quad m\in\NN,
\]
(see \cite{AAW2021b}*{Theorem 5.1}). We will find almost greedy bases for which this estimate is sharp.

For convenience in the notation, from now on given $0<r\le \infty$ we will put $\pot_r=(m^{1/r})_{m=1}^\infty$.
\begin{proposition}\label{prop:ExistenceAG}
Let $0<p<1$ and let $\varphi\colon[0,\infty) \to [1,\infty)$ be a concave increasing function.
Let $\YY$ be one of the spaces $D_{p,q}$, $Z_{p,q}$, $Z_{q,p}$ or $B_{p,q}$. Then $\YY$ has an almost greedy asymptotically suppression unconditional Schauder basis $\YB$ with
\[\sudf[\YB,\YY] \approx \pot_q\quad \mbox{ and }\quad \kk_m[\YB,\YY]\approx \varphi^{1/p} (\log m), \quad m\in\NN.\]
\end{proposition}

\begin{proof}
If $\ell_p$ is complemented in $\YY$ we just apply Theorem~\ref{thm:existence} with $\XX=\ell_p\oplus\YY$, $\XB=\DD\oplus\YB$, where $\YB$ is the canonical basis of $\YY$, and $\Sym=\ell_q$. To obtain the result for $B_{p,q}$, we apply Theorem~\ref{thm:existence} with $\XX=(\bigoplus_{n=1}^\infty \ell_p^{2^n})_{\ell_q}$, $\XB=(\bigoplus_{n=1}^\infty \DD^{(2^n)})_{\ell_q}$ and $\Sym=\ell_q$.
\end{proof}
\section{Fundamental functions of almost greedy bases}\label{sec:FF}\noindent
Once we know that a given space has almost greedy bases, it is natural to embark on a quantitative study aimed at determining all possible fundamental functions of that kind of bases. This computational approach is of interest from the applied point of veiw of greedy approximation theory, but also concerns the more abstract structural aspects of functional analysis. Indeed, as the alert reader may be aware of, the estimation of the norm of same-size blocks of a basis played an important role in determining important features of Banach spaces in the golden years of the theory.

In this section, we will determine the fundamental functions of all almost greedy bases of the spaces $Z_{p,q}$, $B_{p,q}$ and $D_{p,q}$ described in \eqref{dpq}, \eqref{zpq}, and \eqref{bpq}. To contextualize the subject, let us mention that the fundamental function of any almost-greedy basis of a $\SL_p$-space, $p\in(0,1]\cup\{2\}$, grows as $\pot_p$; furthermore, any quasi-greedy basis of such a space is democratic (see \cites{Woj2000,DSBT2012,AAW2021}). In contrast, since for $p\in(1,2)\cup(2,\infty)$, the space $L_p$ has almost greedy bases whose fundamental function grows as $\pot_2$ \cite{Nielsen2007}, the above-mentioned result on uniqueness of fundamental function does not hold for that range of values of $p$. As a matter of fact, any fundamental function of an almost greedy basis of a $\SL_p$-space, $1<p<\infty$, grows as either $\pot_p$ or $\pot_2$ \cite{Ansorena2022}. This applies, in particular, to the spaces $D_{p,2}$, $B_{p,2}$, $B_{2,p}$, $Z_{p,2}$ and $Z_{2,p}$, $1<p<\infty$. In this section, we extend the last result to all locally convex spaces $Z_{p,q}$, $B_{p,q}$ and $D_{p,q}$, thus complementing \cite{AADK2019b}*{Examples 4.6(i) and 4.16(ii)}. We also complement Theorem~\ref{prop:ExistenceAG} by determining all possible fundamental functions of the almost greedy bases of all nonlocally convex spaces $Z_{p,q}$, $B_{p,q}$ and $D_{p,q}$.

We supplement our review of spaces in which the behavior of the fundamental functions of almost greedy bases is known by noting that all quasi-greedy bases of Hardy spaces $H_p(\Disc^d)$, $0<p<1$, $d\in\NN$, are democratic, with fundamental function equivalent to $\pot_p$ \cite{AABe2022}.

For broader applicability, and also because they are a pivotal ingredient in our study, we deal with squeeze-symmetric biorthogonal systems $\XB=(\xx_n,\xx_n^*)_{n=1}^\infty$ instead of almost greedy bases. The fundamental function of $\XB$ will be the fundamental function of its first component, $\XB_0=(\xx_n)_{n=1}^\infty$. We emphasize that we do not even assume $\XB_0$ to be a basic sequence. Also, we would like to draw the reader's attention to the fact if $(\xx_n,\xx_n^*)_{n=1}^\infty$ is a squeeze-symmetric biorthogonal system in $\XX$, and $(\xx_n)_{n=1}^\infty$ spans a subspace $\YY$ of $\XX$, then the biorthogonal system $(\yy_n,\xx_n^*|_\YY)$ is squeeze-symmetric in $\YY$, but the converse does not hold in general. In this section, we will make heavy use of orthogonal systems that do not cover the whole space.

We start by recording some preparatory lemmas.

\begin{lemma}\label{lem:BFF}
Suppose that the fundamental function of a basis $\XB$ of a quasi-Banach space $\XX$ is bounded. Then, $\XX\simeq c_0$.
\end{lemma}

\begin{proof}
It is a ready consequence of \cite{AABW2021}*{Corollary 2.4}.
\end{proof}

Given $0<r\le 1$, the \emph{$r$-Banach envelope} of a quasi-Banach space $\XX$ consists of a $r$-Banach space $\XX_{c,r}$ together with a linear contraction $J_{r,\XX}\colon\XX \to \XX_{c,r}$ satisfying the following universal property: for every $r$-Banach space $\YY$ and every linear contraction $T\colon\XX \to\YY$ there is a unique linear contraction $T_{c,r}\colon \XX_{c,r} \to \YY$ such that $T_{c,r}\circ J_{r,\XX}=T$. We refer the reader to \cite{AACD2018} for more information about this notion, which generalizes the definition of Banach envelope introduced independently by Peetre and Shaphiro in the seventies \cites{Peetre1974,Shapiro1977}. Here, we will take adavantage of the fact that the $r$-Banach envelopes of $Z_{p,q}$, $B_{p,q}$ and $D_{p,q}$ are $Z_{c_r(p),c_r(q)}$, $B_{c_r(p),c_r(q)}$ and $D_{c_r(p),c_r(q)}$ respectively, where $c_r(s)=\max\{r,s\}$ for all $s\in(0,\infty]$.

\begin{lemma}\label{lem:Env}
Suppose that the fundamental function of a squeeze-symmetric basis of a quasi-Banach space $\XX$ grows as $\pot_p$, $0<p<1$. Then, for every $0<p<q\le 1$, the $q$-Banach envelope of $\XX$ is isomorphic to $\ell_q$.
\end{lemma}

\begin{proof}
It follows from \cite{AABW2021}*{Proposition 10.12}.
\end{proof}

One of the advantages of dealing with biorthogonal systems instead of bases is the possibility to pass from a given biorthogonal system to another that inherits the properties of the original one while being easier to handle.
\begin{lemma}\label{lem:SSInhirited}
Suppose that a biorthogonal system $\XB=(\xx_n,\xx_n^*)_{n=1}^\infty$ in a quasi-Banach space $\XX$ is squeezed between symmetric sequence spaces $\Sym_1$ and $\Sym_2$. Let $(n_k)_{k=1}^\infty$ be an increasing sequence in $\NN$. Then the biorthogonal systems
\[
\YB=(\xx_{n_k},\xx_{n_k}^*)_{k=1}^\infty
\]
and
\[
\ZB=\left( \frac{\xx_{2n-1}-\xx_{2n}}{\sqrt{2}}, \frac{\xx_{2n-1}^*-\xx_{2n}^*}{\sqrt{2}}\right)_{n=1}^\infty
\]
are also squeezed between $\Sym_1$ and $\Sym_2$. In particular, if $\XB$ is squeeze-symmetric, so are $\YB$ and $\ZB$, and $\sudf[\XB,\XX] \approx \sudf[\YB,\XX]=\sudf[\ZB,\XX]$.
\end{lemma}

\begin{proof}
It follows from the fact that symmetric bases are unconditional, spreading, and equivalent to its square.
\end{proof}

The principle of small perturbations (see, e.g., \cite{AlbiacKalton2016}*{Theorem 1.3.9}) is an important stability result that is used to alter bases while retaining its essential features. Since a version of this principle does not seem to be available in the literature for biorthogonal systems in quasi-Banach spaces, we next record it for the sake of expositional ease.
\begin{lemma}
Let $(\xx_n,\xx_n^*)_{n=1}^\infty$ be a biorthogonal system in an $r$-Banach space $\XX$, $0<r\le 1$. Suppose that a sequence $(\yy_n)_{n=1}^\infty$ in $\XX$ satisfies
\[
\sum_{n=1}^\infty \norm{\xx_n-\yy_n}^r \norm{\xx_n^*}^r<1.
\]
Then $(\xx_n)_{n=1}^\infty$ and $(\yy_n)_{n=1}^\infty$ are congruent, i.e., there is an isomorphism $S\colon\XX\to\XX$ such that $S(\xx_n)=\yy_n$ for all $n\in\NN$.
\end{lemma}

\begin{proof}
The linear map $T\colon\XX\to\XX$ given by
\[
T(f)=\sum_{n=1}^\infty \xx_n^*(f) (\xx_n-\yy_n), \quad f\in\XX,
\]
satisfies $\norm{T}<1$. Therefore, $S=\Id_\XX-T$ is invertible and the proof is over.
\end{proof}

\begin{lemma}\label{lem:SDinSS}
Let $\LL$ be a minimal quasi-Banach lattice over $\NN$ (or over a countable set). Let $(\XX_n)_{n=1}^\infty$ be a sequence of finite-dimensional quasi-Banach spaces whose moduli of concavity are uniformly bounded. Let $\XB=(\xx_n,\xx_n^*)_{n=1}^\infty$ be a biorthogonal system in
\[
\XX=\left( \bigoplus_{n=1}^\infty \XX_n\right)_\LL.
\]
Let the support of $f=(f_j)_{j=1}^\infty\in\XX$ be the set
\[
\{n\in\NN \colon f_j\not=0\}.
\]
Suppose that $\XB$ is squeeze-symmetric. Then there is a disjointly supported squeeze-symmetric sequence $\YB$ such that
\[
\udf[\XB,\XX] \approx \udf[\YB,\XX].
\]
\end{lemma}

\begin{proof}
Since the unit ball of $\XX_n$ is pre-compact, Cantor's diagonal technique yields a subsequence of $(\xx_n)_{n=1}^\infty$ such that $(\xx_n(j))_{n=1}^\infty$ converges for every $j\in\NN$. Set $\yy_n=\xx_{2n-1}-\xx_{2n}$ for $n\in\NN$. Clearly, $\lim_n \yy_n(j)=0$ for all $j\in\NN$. The gliding hump technique yields a subsequence $(\zz_n)_{n=1}^\infty$ of $(\yy_n)_{n=1}^\infty$ and a disjointly supported sequence $(\uu_n)_{n=1}^\infty$ such that $\norm{\zz_n-\uu_n}$ is arbitrarily close to zero. By the principle of small perturbations, $(\uu_n)_{n=1}^\infty$ is congruent to $(\zz_n)_{n=1}^\infty$. An application of Lemma~\ref{lem:SSInhirited} puts an end to the proof.
\end{proof}

The optimal $r$ such that any of the spaces $Z_{p,q}$, $B_{p,q}$ or $D_{p,q}$ satisfies an upper $r$-estimate is $\min\{p,q\}$, and the optimal $r$ such that any of them satisfies a lower $r$-estimate is $\max\{p,q\}$.

\begin{proposition}\label{prop:SDandLE}
Let $\XX$ be a minimal quasi-Banach lattice over a discrete set. Suppose that $\XX$ satisfies an upper $p$-estimate and a lower $q$-estimate for some $0<p\le q\le \infty$. If $\XB$ is a squeeze-symmetric sequence in $\XX$, then:
\begin{enumerate}[label=(\roman*),leftmargin=*]
\item\label{a:SDandLE} $\pot_q \lesssim \sudf[\XB,\XX] \lesssim \pot_p$.
\item\label{b:SDandLE} If $p>1$, the sequence $\sudf[\XB,\XX]$ has the URP.
\item If $q<\infty$, the sequence $\sudf[\XB,\XX]$ has the LRP.
\end{enumerate}
\end{proposition}

\begin{proof}
An application of Lemma~\ref{lem:SDinSS} gives
\[
r^{1/q} \sudf[\XB,\XX](m) \lesssim \sudf[\XB,\XX](rm) \lesssim r^{1/p} \sudf[\XB,\XX](m), \quad r, \, m\in\NN.\qedhere
\]
\end{proof}

We are now in a position to prove the main result of this section.

\begin{theorem}\label{thm:LC}
Let $0< p, q \le \infty$. If $\XB$ is a squeeze-symmetric sequence in $Z_{p,q}$ there is $r\in\{p,q\}$ such that
\[
\sudf[\XB,\XX](m)\approx \pot_r(m),\quad m\in\NN.
\]
\end{theorem}

\begin{proof}
By Lemma~\ref{lem:SDinSS}, we can suppose that $\XB=(\xx_n)_{n=1}^\infty$ is disjointly supported. Let $r$ be such that $Z_{p,q}$ is an $r$-Banach space. Set
\[
C= \sup_{n\in\NN}\norm{\xx_n^*},
\]
where $\XB^*=(\xx_n^*)_{n=1}^\infty$ in $Z_{p,q}^*$ is such that the biorthogonal system $(\xx_n,\xx_n^*)_{n=1}^\infty$ is squeeze-symmetric. Given $N\in\NN$, the map
\[
P_N\colon Z_{p,q} \to Z_{p,q}
\]
will be the canonical projection onto $\YY=\ell_q^{(N)}(\ell_p)$.

We distinguish two complementary cases.

\noindent \textbf{Case 1: $q\ge p$.} \\
Suppose that for $N\in\NN$, $\delta>0$, and $n_0\in\NN$, there is $n\in\NN$, $n\geq n_0$, such that $\norm{P_N(\xx_n)}<\delta$. Then, given $\bm{\delta}=(\delta_n)_{n=1}^\infty$ in $(0,\infty)$, we recursively construct increasing sequences $(N_k)_{k=1}^\infty$ and $(n_k)_{k=1}^\infty$ such that, if
\[
\yy_k=(P_{N_{k+1}}-P_{N_k})(\xx_{n_k}) , \quad k\in\NN,
\]
then $\norm{\xx_{n_k}-\yy_k} \le \delta_k$ for all $k\in\NN$. By the principle of small perturbations, choosing $\bm{\delta}$ suitably we obtain that $(\xx_{n_k})_{k=1}^\infty$ is equivalent to $(\yy_k)_{k=1}^\infty$. In turn, $(\yy_k)_{k=1}^\infty$ is equivalent to the unit vector system of $\ell_q$.

Suppose now that there are $N\in\NN$ and $n_0\in\NN$ such that
\[
\delta:=\inf_{n\ge n_0} \norm{P_N(\xx_n)}>0.
\]
Then $\sldf[\XB,\XX](m) \ge \delta m^{1/p}$ for all $m\in\NN$. By Proposition~\ref{prop:SDandLE}\ref{a:SDandLE}, $\sudf[\XB,\XX] \approx \pot_p$.

\medskip \noindent\textbf{Case 2: $q<p$.} \\
Suppose that for $N\in\NN$, $0<\delta<1$, and $n_0\in\NN$ there are $n>m>n_0$ such that
\begin{equation}\label{eq:C2}
\abs{ (\xx_n^*-\xx_m^*)(P_N(\xx_n-\xx_m))}<\delta.
\end{equation}
Then, since $(\xx_n^*-\xx_m^*)(\xx_n-\xx_m)=2$ and $\norm{\xx_n^*-\xx_m^*} \le 2 C$, \eqref{eq:C2} implies that
\[
\norm{ \xx_n-\xx_m -P_N(\xx_n-\xx_m)}=\sup_{M>N} \norm{(P_M-P_N)(\xx_n-\xx_m)} \ge \frac{2-\delta}{2C}.
\]
Fix $0<\lambda<1/C$. We recursively construct increasing sequences $(N_k)_{k=1}^\infty$ and $(n_k)_{k=1}^\infty$ in $\NN$ such that the sequence $\YB=(\yy_k)_{k=1}^\infty$ given by
\[
\yy_k=( P_{N_{k+1}}-P_{N_k})(\xx_{n_{2k}} - \xx_{n_{2k-1}} )
\]
satisfies $\norm{\yy_k} \ge \lambda$ for all $k\in\NN$. We infer that $\sudf[\XB,\XX](2m) \ge \lambda m^{1/q}$ for all $m\in\NN$. Consequently, by Proposition~\ref{prop:SDandLE}\ref{a:SDandLE}, $\sudf[\XB,\XX]\approx \pot_q$.

Now, we suppose that instead there are $N\in\NN$, $\delta>0$ and $n_0$ such that
\[
\abs{(\xx_n^*-\xx_m^*)(P_N(\xx_n-\xx_m))}\ge \delta, \quad n >m>n_0.
\]

Given $(f,f^*)\in\XX\times\XX^*$, the sequences
\[
f^*(P_N(\xx_n)), \quad\mbox{ and }\quad P_N^*(\xx_n^*)(f)=\xx_n^*(P_N(f)), \quad n\in\NN,
\]
are bounded. Then there is an increasing map $\alpha\colon\NN\to\NN$ such that the sequences
\[
f^*(P_N(\xx_{\alpha(n)})), \quad \xx_{\alpha(n)}^*(P_N(f)), \quad n\in\NN,
\]
converge. Combining this fact with Cantor's diagonal argument we obtain an increasing map $\beta\colon\NN\to\NN$ such that the sequences
\[
\xx_j^*(P_N(\xx_{\beta(n)})), \quad\mbox{ and }\quad \xx_{\beta(n)}^*(P_N(\xx_j)), \quad n\in\NN,
\]
converge for every $j\in\NN$. Hence, the biorthogonal system $(\zz_n,\zz_n^*)_{n=1}^\infty$ given by
\[
\zz_n= \frac{\xx_{\beta(2n-1)}-\xx_{\beta(2n)}}{\sqrt{2}}, \quad \zz_n^*=\frac{\xx_{\beta(2n-1)}^*-\xx_{\beta(2n)}^*}{\sqrt{2}}
\]
satisfies
\[
\lim_n \zz_j^*(P_N(\zz_n))=\lim_n \zz_n^*(P_N(\zz_j))=0, \quad j\in\NN.
\]
Pick a non-increasing sequence $(\epsilon_n)_{n=1}^\infty$ of positive scalars such that
\[
\sum_{n=2}^\infty \epsilon_n \le \delta/4.
\]
There is an increasing map $\gamma\colon\NN\to\NN$ such that
\[
\abs{ \zz_{\gamma(j)}^*(P_N(\zz_{\gamma(n)}))} \le \epsilon_{\max\{j,n\}}, \quad j\not=n.
\]
Since by Lemma~\ref{lem:SSInhirited} we can replace $(\xx_n)_{n=1}^\infty$ with $(\zz_{\gamma(n)})_{n=1}^\infty$, we may safely assume that
\begin{itemize}
\item $\abs{\xx_n^*(P_N(\xx_n))}\ge \delta/2$ for all $n\in\NN$,

\item $\abs{\xx_j^*(P_N(\xx_n))}\le \epsilon_{\max\{j,n\}}$ for all $(j,n)\in\NN^2$ with $j\not=n$, and

\item $\XB$ is disjointly supported.

\end{itemize}
Since $\ZB:=(P_N(\xx_n))_{n=1}^\infty$ is semi-normalized and disjointly supported, and $\YY$ is lattice isomorphic to $\ell_p$, we conclude that $\ZB$ is equivalent to the standard $\ell_p$-basis. For $A\subseteq\NN$ finite and $n\in A$ we have
\begin{align*}
\abs{\xx_n^*(\Ind_A[\ZB,Z_{p,q}])}
&\ge \abs{\xx_n^*(P_N(\xx_n))}-\sum_{j\in A\setminus\{n\}}{\xx_n^*(P_N(\xx_j))}\\
&\ge \frac{\delta}{2}-\sum_{j\in A\setminus\{n\}}\epsilon_{\max\{j,n\}}\\
&\ge \frac{\delta}{2}-\sum_{j\in \NN\setminus\{n\}}\epsilon_{\max\{j,n\}}\\
&=\frac{\delta}{2}-(n-1)\epsilon_n + \sum_{j=n+1}^\infty \epsilon_j\\
&\ge\frac{\delta}{2}- \sum_{j=2}^\infty \epsilon_j\ge \frac{\delta}{4}.
\end{align*}
For $m\in\NN$, choose $A_m\subseteq\NN$ with $\abs{A_m}=m$. By squeeze-symmetry,
\[
\sudf[\XB,\XX](m) \lesssim \norm{ \Ind_{A_m}[\ZB,Z_{p,q}]) } \approx m^{1/p}, \quad m\in\NN.
\]
An application of Proposition~\ref{prop:SDandLE}\ref{a:SDandLE} puts an end to the proof.
\end{proof}

\begin{remark}
If $(\xx_n,\xx_n^*)_{n=1}^\infty$ is a squeeze-symmetric biorthogonal system of $D_{p,q}$, we can prove Theorem~\ref{thm:LC} easily. Indeed, we can assume without loss of generality that $(\xx_n)_{n=1}^\infty$ is disjointly supported and that $0<p\le q$. Write $\xx_n=(\yy_n,\zz_n)$, $n\in\NN$.

If $\delta:=\inf_n \norm{\zz_n}_p>0$, then $\sldf[\XB,D_{p,q}](m) \ge \delta m^{1/p}$ for all $m\in\NN$. Consequenty, by Lemma~\ref{prop:SDandLE}\ref{a:SDandLE}, $\sldf[\XB,D_{p,q}] \approx \pot_p$.

Otherwise, the principle of small perturbations yields a subsequence of $\XB$ which is equivalent to a subsequence of
$\YB=(\yy_n)_{n=1}^\infty$. Since $\YB$ is semi-normalized and disjointly supported, $\YB$ is equivalent to the canonical $\ell_q$-basis. We infer that $\sudf[\XB,D_{p,q}] \approx\pot_q$.
\end{remark}

Theorem~\ref{thm:LC} allows us to significantly advance the understanding of the fundamental functions of almost greedy bases of matrix spaces $Z_{p,q}$. Notice that, since $D_{p,q}$ and $B_{p,q}$ are complemented subspaces of $Z_{p,q}$, Theorem~\ref{thm:LC} also applies to these families of mixed-norm spaces. As we will see in the upcoming Sections~\ref{SecCase1} and \ref{SecCase2}, the spaces $Z_{p,q}$ and $D_{p,q}$ have a similar almost greedy bases structure, while the almost greedy bases structure of the spaces $B_{p,q}$ deviates slightly from their pattern.

\subsection{Fundamental functions of almost greedy bases of Besov spaces $Z_{p,q}$ and mixed-norm spaces $D_{p,q}$, $0<p,q<\infty$.}\label{SecCase1}
\begin{enumerate}[label=(\alph*),leftmargin=*]
\item\label{ex:Zpq:a} Let $1\le p,q<\infty$. By Theorem~\ref{thm:LC}, the fundamental function of any squeeze-symmetric basis of $Z_{p,q}$ and $B_{p,q}$ is equivalent to either $\pot_p$ or $\pot_q$. Conversely, by \cite{AADK2019b}*{Examples 4.6(ii) and 4.16(i)}, given $r\in\{p,q\}$, the space $Z_{p,q}$ has almost greedy bases whose fundamental functions grow as $\pot_r$.

\item Let $1\le p<\infty$. Since none of the spaces $Z_{0,p}$, $Z_{p,0}$ or $D_{0,p}$ is isomorphic to $c_0$,
combining Theorem~\ref{thm:LC} with Lemma~\ref{lem:BFF} gives that the fundamental function of any squeeze-symmetric basis of these spaces grows as $\pot_p$. Conversely, by \cite{AADK2019b}*{Example 4.16(i)}, $Z_{0,p}$, $Z_{p,0}$ and $D_{0,p}$ have almost greedy bases whose fundamental functions grow as $\pot_p$.

\item Let $0<q<1<p<\infty$. By Proposition~\ref{prop:ExistenceAG}, $Z_{p,q}$, $Z_{q,p}$ and $B_{p,q}$ have almost greedy bases whose fundamental functions grow as $\pot_p$. Reciprocally, by Theorem~\ref{thm:LC}, any squeeze-symmetric basis of any of these spaces has fundamental function equivalent to $\pot_r$ for $r\in\{p,q\}$. Since the Banach envelope of these spaces is not isomorphic to $\ell_1$, an application of Lemma~\ref{lem:Env} gives that $r=p$.

\item Let $0<p<1$. Since the Banach envelope of $Z_{p,0}$, $Z_{0,p}$ or $D_{p,0}$ is not isomorphic to $\ell_1$, combining Theorem~\ref{thm:LC}, Lemma~\ref{lem:BFF} and Lemma~\ref{lem:Env} gives that no of these spaces has a squeeze-symmetric basis.

\item Let $0<p,q<1$, $p\not=q$. By \cite{AABe2022}, neither $Z_{p,q}$ nor $D_{p,q}$ has a squeeze-symmetric basis.

\item Let $0<p< 1$. It seems to be unknown whether $Z_{p,1}$, $Z_{1,p}$ or $D_{p,1}$ have an almost greedy basis. If it were the case, its fundamental function would be equivalent to $\pot_1$. Indeed, if $\XB$ were a squeeze-symmetric basis of $Z_{p,1}$, $Z_{1,p}$, or $D_{p,1}$, its fundamental function would be equivalent to $\pot_r$ for $r\in\{1,p\}$ by Theorem~\ref{thm:LC}. If $r=p$, given $p<q<1$ the $q$-Banach envelope of $Z_{p,1}$, $Z_{1,p}$ or $D_{p,1}$ would be on one hand isomorphic to $\ell_q$ by Lemma~\ref{lem:Env}, and, on the other hand, isomorphic to $Z_{q,1}$, $Z_{1,q}$, or $D_{p,1}$, respectively. Since none of this spaces is isomorphic to $\ell_q$, we would reach a contradiction.
\end{enumerate}

\subsection{Fundamental functions of almost greedy bases of Besov spaces $B_{p,q}$, $0<p,q<\infty$.}\label{SecCase2}
Finally, we analyze the fundamental functions of almost greedy bases of Besov spaces $B_{p,q}$. The situation is as follows.
\begin{enumerate}[label=(\alph*),leftmargin=*]
\item Let $1<q<\infty$ and $0<p\le \infty$. By Lemma~\ref{lem:SDinSS}, the fundamental function of any squeeze-symmetric basis of $B_{p,q}$ grows as $\pot_q$. Conversely, by Proposition~\ref{prop:ExistenceAG} and \cite{AADK2019b}*{Examples 4.6(ii)}, $B_{p,q}$ has almost greedy bases whose fundamental functions grow as $\pot_q$.

\item Let $1\le p\le \infty$. By Lemma~\ref{lem:SDinSS}, the fundamental function of any squeeze-symmetric basis of $B_{p,1}$ grows as $\pot_1$. Conversely, by \cite{AADK2019b}*{Example 4.16}, $B_{p,1}$ has almost greedy bases whose fundamental function grows as $\pot_1$.

\item Let $0<p<\infty$. By Lemma~\ref{lem:SDinSS} and Lemma~\ref{lem:BFF}, $B_{p,0}$ has no squeeze-symmetric basis.

\item Let $0<q<1$, and $0<p\le \infty$. If $q<p$, we pick $p<r<\min\{1,p\}$ so that the $r$-Banach envelope of $B_{p,q}$ is not isomophic to $\ell_r$. Combining Lemma~\ref{lem:SDinSS} with Lemma~\ref{lem:Env}, gives that $B_{p,q}$ has no squeeze-symmetric basis. If $p<q$, $B_{p,q}$ has no squeeze-symmetric basis either by \cite{AABe2022}. Summing up, $B_{p,q}$ has no squeeze-symmetric basis unless $p=q$, so that $B_{p,q}=\ell_q$.

\item The spaces $B_{p,1}$ for $0<p<1$ do not possess greedy bases because they have a unique unconditional basis up to a permutation and the canonical basis of those spaces is (unconditional but) not democratic. However, we do not know whether $B_{p,1}$, has an almost greedy basis. If it were the case, its fundamental function would be equivalent to $\pot_1$.
\end{enumerate}

\end{document}